\documentclass[11pt]{amsart}
\usepackage{latexsym}

\def\eea{\end{eqnarray*}}

\newtheorem{thm}{Theorem}[section]
\newtheorem{prop}[thm]{Proposition}

\newtheorem{lemma}[thm]{Lemma}

\begin{document}
\renewcommand{\theequation}{\thesection.\arabic{equation}}

\title[3-d $m$-quasi Einstein manifolds with two Ricci eigenvalues]{Three dimensional $m$-quasi Einstein manifolds with degenerate Ricci tensor}

\author{Jongsu Kim and Jinwoo Shin}
\date{\today}

\address{Dept. of Mathematics, Sogang University, Seoul, Korea}
\email{jskim@sogang.ac.kr}

\address{Dept. of Mathematics, Sogang University, Seoul, Korea}
\email{shinjin@sogang.ac.kr}

\thanks{This research was supported by Basic Science Research Program through the National Research Foundation of Korea(NRF) funded by the Ministry of Science, ICT \& Future Planning (NRF-2017R1A2B4004460)}



\begin{abstract}
In this article we give a classification of three dimensional $m$-quasi Einstein manifolds  with two distinct Ricci-eigen values. Our study provides explicit description of local and complete metrics and potential functions.
%We find that these are closely related to Kobayashi's equation  in \cite{Ko}.
 We also describe the associated warped product Einstein manifolds in detail.

For the proof we present a Codazzi tensor on any three dimensional $m$-quasi Einstein manifold
and use geometric properties of the tensor which help to analyze the  $m$-quasi Einstein equation effectively. A technical advance over preceding studies is made by resolving the case
when the gradient $\nabla f$ of the potential function  is not a Ricci-eigen vector field.

 %In particular, we manage to analyze the case when the gradient of the potential function is not a Ricci-eigen vector field, while in preceding study of

%We also find an example of a Riemannian manifold which is pertinent to understand a Besse's statement on Codazzi tensors with exactly two distinct eigenvalues \cite[16.12]{Be}
\end{abstract}

\maketitle

%\vfill
%\pagebreak
\setcounter{section}{0}
\setcounter{equation}{0}

\section{Introduction}

An $m$-quasi Einstein (denoted by $m$-QE, below)  manifold is a Riemannian manifold $(M, g, f)$ with a smooth function $f$ which satisfies
\begin{equation}
Rc+\nabla df - \frac{1}{m}df\otimes df=\lambda g \label{quasidef}
\end{equation}
for a real constant $\lambda$, an integer $m>1$ and $Rc$ denotes the Ricci tensor of $(M,g)$.
This definition is coined so that the  warped product manifold $(M \times F, \ g_E:=g + e^{-\frac{2f}{m}} g_F)$, with an $m$-dimensional Einstein manifold $(F, g_F)$, is Einstein with  $Rc_{g_E} =  \lambda  g_E$.
 If $f$ is constant, then $g$ is Enstein and we call this space $(M, g, f={\rm const})$ {\em trivial}.
One or two dimensional $m$-QE  manifolds have been thoroughly understood early and presented in the section 9.J of \cite{Be}. An $m$-QE manifold is related to a gradient Ricci soliton, as the latter corresponds to the case $m=\infty$.

Recent progress in understanding gradient Ricci solitons further motivated  the study of $m$-QE manifolds. When $M$ is compact and $\lambda \leq 0$, $\ (M, g, f)$ is proved to be trivial \cite{KK,  HPW}.
 In dimension$\geq 3$ various characterizations, triviality or rigidity  have been studied under certain geometric conditions. Here we mention only those most related to our work. If $(M, g)$ is complete locally conformally flat, then it is shown to be locally a warped product over an interval near a regular point of $f$, \cite{CMMR}. The condition of harmonic Weyl tensor on $(M, g)$ was studied in \cite{HPW} and \cite{SJ},  while half-conformally-flat condition in \cite{Cat}, certain Bach flat condition  in \cite{CH, RR} and other rigidity results in \cite{CSW, Wa1}.

We note that not much is known about three dimensional (denoted by $3$-d, below)
 $m$-QE manifolds, though some characterizations were known in certain cases, for instance  when $(M, g)$ is locally conformally flat as mentioned above,  or homogeneous \cite{BRS}.

In this paper, we shall characterize 3-d $m$-QE manifolds with two distinct Ricci eigenvalues.
Our work is a generalization of Theorem 9.119 in \cite{Be} and \cite{Pa}. As locally conformally flat 3-d $m$-QE manifolds have degenerate Ricci  tensor, one may view that our work  also generalizes the 3-dimensional part of the work in \cite{CMMR}.
We prove;
\begin{thm} \label{maint1}
  Let $(M,g,f)$ be a 3-dimensional $m$-QE manifold with Ricci-eigen values $\lambda_1\neq\lambda_2=\lambda_3$. Then near each point in the open dense subset $\{ x   | \ \nabla f(x) \neq 0 \}$ of $M$, there exist local coordinates $(x_1,x_2,x_3)$ in which $(g,f)$ can be one of the following; for constants $ a \neq 0, \rho,\mu,k$ and functions $p:=p(x_3)$, $\eta:=\eta(x_1)$, $\tau:=\tau(x_1)$ and $h:=h(x_1)$,
\begin{enumerate}
  \item $g=dx_1^2+\eta^2(p')^2 dx_2^2+\eta^2 dx_3^2$  with   $ \ (p')^2+\rho p^2+\frac{2a}{m-1}p^{1-m}=\frac{\mu}{m-1}$ and $(\eta^{'})^2+\frac{\lambda}{m+2}\eta^2=\rho.$ % And $w=p(x_3)\eta(x_1)$.
     And $f=-m \cdot \ln \{ p \cdot \eta\}$.
     \smallskip
  \item  $g=p^2 dx_1^2+(p')^2 dx_2^2+dx_3^2$ with $(p')^2+\frac{\lambda}{m+2}p^2+\frac{2a}{m}p^{-m}=k$.
  %And $w=q(x_1)p(x_3)$
     And $f=-m \cdot \ln \{\tau \cdot p\}$  with
$(\tau^{'})^2+k(\tau)^2=\frac{\mu}{m-1}$.
     \smallskip
\item $g= dx_1^2 +      h(x_1)^2 \tilde{g},$ where
 $\tilde{g}$ is a metric of constant curvature on a domain with $x_2, x_3$ coordinates. $g$ is locally conformally flat.
\end{enumerate}

\end{thm}
In this article we do not aim to study the case (3) of locally conformally flat spaces. Indeed, this study includes B\"{o}hm's work \cite{Bo} and may be still a big subject for research.

By Theorem \ref{maint1} one actually understands 3-d $m$-QE manifolds with one or two distinct Ricci-eigen values; see \cite{HPW} for a summary of $m$-QE manifolds which are also Einstein.

We present {\it complete} warped product Einstein manifolds associated to the m-QE manifolds of Theorem \ref{maint1}.  Detailed descriptions are in Section 8.
\begin{thm} \label{maint2}
 Any complete warped product Einstein manifold $(N, g_E)$ associated to a three dimensional m-QE manifold $(M, g)$  with exactly two distinct Ricci eigenvalues on an open dense subset of $M$
 is one of the followings; for an m-dimensional complete Einstein manifold $(F, g_F)$ with $Rc_{g_F} =  \mu g_F$,
\begin{equation} \label{440c1}
 {\rm (i)}   \ \   g_E  = dx_1^2+(\eta(x_1))^2 \{ (p'(x_3))^2dx_2^2+dx_3^2+ p^2 g_F \}, \hspace{4cm} \nonumber
\end{equation}
where $(\eta^{'})^2+\frac{\lambda}{m+2}\eta^2=\rho $ and $ \ \ (p^{'})^2+\rho p^2+\frac{2a}{m-1}p^{1-m}=\frac{\mu}{m-1}$, in one of the three cases:

{\rm (i-1)}  $\lambda<0$, $a>0$, $\rho  \leq 0$,

{\rm (i-2)} $\lambda<0$, $a<0$, $\rho <0$, $\mu\leq (m+1)\rho (\frac{a}{ \rho })^{\frac{2}{m+1}}$, $p(0)>(\frac{a}{ \rho })^{\frac{1}{m+1}}$,

{\rm (i-3)} $\lambda=\rho =0$, $a>0$.

\smallskip
\noindent And $g_E$ is on (a non-compact quotient  of) $\mathbb{R}^3\times F$.
\begin{equation} \label{440b1}
  {\rm (ii)}  \ \ g_E  = (p(x_3))^2 (dx_1^2 + \tau(x_1)^2 g_F )+     (p^{'})^2 dx_2^2+ dx_3^2,\hspace{5.2cm} \nonumber
\end{equation}
where $(p')^2+\frac{\lambda}{m+2}p^2+\frac{2a}{m}p^{-m}=k$ and $   \ \ ({\tau}')^2  + k({\tau})^2=\frac{\mu}{m-1}$,
in one of the two cases:

{\rm (ii-1)} $a>0$, $\lambda\leq 0$,

{\rm (ii-2)} $a<0$, $\lambda<0$, $\ p(0)>\left\{\frac{a(m+2)}{ \lambda}\right\}^{\frac{1}{m+1}}$, $ \ \ k\leq \frac{\lambda}{m}\left\{\frac{a(m+2)}{ \lambda}\right\}^{\frac{2}{m+1}}.$

\noindent Depending on conditions, $g_E$ can be on (a non-compact quotient of) $ \mathbb{R}^3\times F $, $ \ \mathbb{S}^{m+1} \times  \mathbb{R}^{2},  \ {\rm or} \ \mathbb{R}^{m+3}$.

 \medskip
\noindent {\rm (iii)} $g_E= dx_1^2 +      h(x_1)^2 \tilde{g} + e^{-\frac{2f}{m}} g_F$ where
 $\tilde{g}$ is a Riemannian metric of constant curvature on a $2$-dimensional domain with $x_2, x_3$ coordinates.

\end{thm}

The starting point in proving Theorem \ref{maint1} and  Theorem \ref{maint2} is to observe the existence of a Codazzi tensor $  \mathcal{C}$ of the form $\phi Rc + \psi g$ associated to the 3-d $m$-QE equation. If $\lambda_1 \neq \lambda_2$ are Ricci-eigenvalues with multiplicity $1$ and $2$ respectively, letting $\mu_i:=\phi \lambda_i + \psi$, $i=1,2$, we have $ \mathcal{C}$-eigenvalues $\mu_1 \neq  \mu_2$ with the same eigenspaces as $\lambda_1, \lambda_2$ respectively.
Then the two dimensional $\lambda_2$-eigenspace of  $ \mathcal{C}$ forms an integrable and umbilic distribution \cite{De}.
This provides a good coordinate system, in which we can express
$g$ nicely. Some arguments and computations then reduce the  $m$-QE equation to solvable ordinary differential equations in possible cases.

When the metric is locally conformally flat, or more generally when it is of harmoic Weyl curvature or Bach-flat in higher dimension, the gradient $\nabla f$ of the potential function is readily a Ricci-eigen vector field. This fact crucially helps to resolve the $m$-QE equation in the aforementioned works. The same is true of the study for gradient Ricci solitons. It is harder to study the case when $\nabla f$  is not Ricci-eigen. So, a novelty of this article is to analyze the case  when $\nabla w$ is not Ricci-eigen, at least in dimension three.

%In the recent study of geometric equations which involve Ricci tensor and the Hessian of the potential function,  such as gradient Ricci solitons and $m$-QE manifolds,

We expect that the method explored for three dimensional $m$-QE manifolds in this article can be applied to other geometric equations.

%As we proceeded our study with the eigenvalue-degenerate Codazzi tensor, we noticed a  structure theorem on such a tensor in \cite{CMM} which might reduce our arguments. But we found an example that does not fit the theorem and present it in the last section.  This would vindicate the necessity of our arguments.

% one may wish to avoid much hassle of computationsby adopting warped product structures and/or totally geodesic foliations, as promised by  a structure theorem on such a tensor in \cite{CMM}.

\bigskip
The paper is organized as follows. In section 2 we present some properties of $m$-QE manifolds and find the wanted Codazzi tensor. Here by  $w=e^{-\frac{f}{m}}$, we transform $m$-QE manifold into (equivalent) $(\lambda,n+m)$-Einstein manifold $(M, g, w)$ satisfying {\rm (\ref{hpwdef})}.
In section 3, we pursue a local description of a $(\lambda,3+m)$-Einstein metric with two distinct Ricci-eigenvalues until we have to divide into four cases, to be treated in the next four sections.
For an {\it adapted} frame vector field $E_i$, $i=1,2,3$, we study in section 4 the case when $\nabla w$ is not a Ricci-eigen vector field and $E_1  \mu_2  \neq 0$,
  in section 5  when  $\nabla w$ is not  Ricci-eigen and $E_1  \mu_2  = 0$,
 in section 6  when $ \nabla w$ is  Ricci-eigen and $\frac{\nabla w}{|\nabla w |} = E_1$, and in section 7
  when $ \nabla w$ is  Ricci-eigen and $\frac{\nabla w}{|\nabla w |} = E_3$.
In section 8, we discuss the associated warped product Einstein manifolds, in particular complete ones,  and finishes the proofs of two main theorems.

%In section 9, we present a 3-d Riemannian manifold $(U, g)$ with a Codazzi tensor $A$ having eigenvalues $\mu_1$ of multiplicity $1$ and  $\mu_2$ of multiplicity $2$. And $(U, g)$ has nowhere a warped-product structure and does not have totally geodesic foliations tangent to $\mu_2$-eigenspaces.

\section{Three dimensional $m$-QE manifold}

 Let $(M,g, f)$ be an $n$ dimensional  $m$-QE manifold.
If we take  $w=e^{-\frac{f}{m}}$
in (\ref{quasidef}), then  we obtain the following equations which is called the $(\lambda,n+m)$-Einstein equation \cite{HPW};
\begin{equation}
  \nabla dw=\frac{w}{m}(Rc-\lambda g). \label{hpwdef}
\end{equation}
Upon this transformation between $w$ and $f$,
it is natural to assume $w>0$ on the interior of $M$
and    $w=0$  on the boundaray $\partial M$.

Although (\ref{quasidef}) and {\rm (\ref{hpwdef})} are essentially equivalent,
we prefer to use the latter for technical reason.

\begin{lemma}\label{kklem}  \cite{KK}, \cite[(2.1)]{HPW} and \cite[(3.12)]{CSW}
  Let $(M,g,w)$ be a $(\lambda,n+m)$-Einstein manifold. It holds that

 {\rm (i)} There is a constant $\mu$  such that
\begin{align}
  \mu=\frac{R+(m-n)\lambda}{m}w^2+(m-1)|\nabla w|^2. \label{consteq}
\end{align}

{\rm (ii)}  $ \   \frac{w}{2(m-1)}\nabla R  = -R(\nabla w, )+ \frac{(n - 1)\lambda - R }{ m-1}  g(\nabla w, ) .$
\end{lemma}
Note that the above (ii) comes from \cite[(3.12)]{CSW} by the relation $w= e^{-\frac{f}{m}}$.
The constant $\mu$ is the Ricci curvature of the fiber $F$ of the warped product Einstein metric over $M$.

\medskip
We shall find a Codazzi tensor associated to a $(\lambda,3+m)$-Einstein manifold. This kind of Codazzi tensors is already found in 3-d static spaces \cite[Example 5]{DS}
and 3-d gradient Ricci solitons \cite{CMM}.

 \begin{lemma} \label{threesolx}
For a $(\lambda,3+m)$-Einstein manifold $(M, g, w)$,
\begin{align}
  \mathcal{C}=w^{m+1}Rc+\frac{w^{m+1}}{2(m-1)}(2\lambda-mR)g
\end{align}
 is a Codazzi tensor.

\end{lemma}

\begin{proof}
The equation {\rm {\rm (\ref{hpwdef})}} in dimension $3$  gives
\begin{align*}
\nabla_i  \nabla_j  \nabla_k w  = \nabla_i  \{ \frac{w}{m} (R_{jk}  - \lambda g_{jk})   \}
 =  \frac{ \nabla_i w}{m}(R_{jk}  - \lambda g_{jk})   +\frac{w}{m}   \nabla_i R_{jk}.
\end{align*}

By Ricci identity,
\begin{align} \label{not01}
-mR_{ijkl} \nabla_l w =m\nabla_i  \nabla_j  \nabla_k w   - m\nabla_j  \nabla_i  \nabla_k w \hspace{3cm}  \\
   =   \nabla_i w  R_{jk} -\nabla_j w R_{ik}-    \nabla_i w \lambda  g_{jk}+\nabla_j w\lambda   g_{ik}  + w ( \nabla_i R_{jk}- \nabla_j R_{ik}). \nonumber
\end{align}

Meanwhile in dimension $3$,
\begin{eqnarray} \label{not02}
R_{ijkl}\nabla_l w = - (R_{ik} \nabla_j w + R_{jl}g_{ik}\nabla_l w - R_{il}g_{jk}\nabla_l w -  R_{jk}\nabla_i w) \\
+ \frac{R}{2}(g_{ik}\nabla_j w - g_{jk}\nabla_i w) \nonumber \hspace{3cm}
\end{eqnarray}
Adding (\ref{not01}) and $m \times $(\ref{not02}), we get
\begin{align*}
   0=& (m+1)R_{jk}\nabla_i w-(m+1)R_{ik} \nabla_j w - mR_{jl}g_{ik}\nabla_l w +m R_{il}g_{jk}\nabla_l w \\
   &+  w \nabla_i R_{jk}- w\nabla_j R_{ik}+(m \frac{R}{2}+ \lambda)g_{ik}\nabla_j w -(m \frac{R}{2}+ \lambda) g_{jk}\nabla_i w.
\end{align*}
  Multiplying this with $w^m$ and putting $  w^{m+1}  \nabla_i R_{jk} + (m+1)w^m\nabla_i w  R_{jk} = \nabla_i (w^{m+1}   R_{jk})$, we get
\begin{align}
0=\nabla_i (w^{m+1}   R_{jk}) - \nabla_j (w^{m+1}   R_{ik})
  +m (R_{il}g_{jk}- R_{jl}g_{ik})w^m \nabla_l w \label{16} \\
\qquad -(m \frac{R}{2}+ \lambda) g_{jk}w^m\nabla_i w+(m \frac{R}{2}+ \lambda)g_{ik}w^m \nabla_j w. \hspace{2cm} \nonumber
 \end{align}
From Lemma \ref{kklem} (ii), we have

$ \frac{-m}{2(m-1)}\nabla _i( R w^{m+1}g_{jk} )
=   mR_{il}g_{jk}w^{m}\nabla_l w
- m\frac{ R}{2} g_{jk}w^{m} \nabla_iw
-\frac{2m}{ m-1} \lambda g_{jk} w^{m}\nabla_i w.$

\noindent With this, (\ref{16}) becomes;
\begin{eqnarray} \label{eth010n}
\nabla_i (w^{m+1}   R_{jk}) - \nabla_j (w^{m+1}   R_{ik})
-\frac{m}{2(m-1)}\nabla _i( R w^{m+1}g_{jk} ) \hspace{1.3cm}\\
 \ \ \ \ +\frac{m}{2(m-1)}\nabla_j( R w^{m+1}g_{ik} )+\frac{ \lambda }{ m-1} \nabla_i (w^{m+1} g_{jk})
- \frac{\lambda }{ m-1} \nabla_j (w^{m+1} g_{ik}). \nonumber
 \end{eqnarray}

 From (\ref{eth010n}),
 $\mathcal{C}$ is a Codazzi tensor.
\end{proof}

Some properties of a Codazzi tensor are explained in \cite{De} or \cite[16.11]{Be};

\begin{lemma} \label{derdlem}
Let $A$ be a Codazzi tensor on a  Riemannian manifold $M$.
In a connected open subset of $M$ where the number of distinct eigenvalues of $A_x$, $x \in M$, is constant,

{\rm (i) } Given distinct eigenfunctions $\lambda, \mu$ of $A$ and local vector fields $v, u_1, u_2$ such that  $A v = \lambda v$, $Au_i = \mu u_i$, it holds that

$ \ \ \ \ \  v(\mu) g(u_1, u_2) = (\mu - \lambda) <\nabla_{u_1} u_2, v > $.

{\rm (ii)} For each eigenfunction $\lambda$, the $\lambda$-eigenspace distribution is integrable and its leaves are totally umbilic submanifolds of $M$.

{\rm (iii)} If $\lambda$-eigenspace $V_{\lambda}$ has dimension bigger than one, the eigenfunction $\lambda$ is constant along the leaves of $V_{\lambda}$.

{\rm (iv)} Eigenspaces of $A$ form mutually orthogonal differentiable distributions.
\end{lemma}

\begin{lemma}\cite{HPW}
  Let $(M,g,w)$ be a $(\lambda,n+m)$-Einstein manifold. Then $g$ and $w$ are real analytic in harmonic coordinates where $w \neq 0$.
\end{lemma}

So, if $w$ is not a constant then $\{\nabla w\neq 0\}$ is open and dense in $M$.
The eigenvalues of   $\mathcal{C}$ are real analytic in an open set  where the number of distinct eigenvalues of  $\mathcal{C}_x$ stays constant and $w \neq 0$.

\medskip
Finally we prove a basic Riemannian geometric property.
\begin{lemma} \label{claim112b8}
Let $(U, g)$ be  a 3-dimensional Riemannian manifold with an orthonormal frame field $E_1, E_2, E_3$ such that
the three distributions $E_{ij}$ spanned pointwise by $E_i$ and $E_j$ for $i \neq j$ are integrable.

\smallskip
Then for each point $p_0$ in $U$, there exists a neighborhood $V$ of $p_0$ in $U$ with coordinates $(x_1, x_2, x_3)$   such that
$g$ can be written on $V$ as
\begin{equation} \label{mtr1a}
 g= g_{11}dx_1^2 +   g_{22}  dx_2^2 +    g_{33} dx_3^2 ,
\end{equation}
 where  $g_{ij}$ are smooth functions of $x_1, x_2, x_3$. And   $E_i =\frac{1}{ \sqrt{g_{ii}} } \frac{\partial }{\partial x_i} $.

\end{lemma}

 \begin{proof}
 The metric $g$ can be written as
 $g= E_1^* \otimes E_1^* + E_2^* \otimes E_2^* +  E_3^* \otimes E_3^* $, where
 $\{E_i^* \}, i=1,2,3,$ is the dual co-frame field of $\{E_i\}$.

As $E_{23}$ is integrable,
we apply the argument of Lemma 4.2 in \cite{Ki} to $\{ c E_1 \ | \ c \in \mathbb{R} \}$ and $E_{23}$, and have a coordinate system $y_1,y_2,y_3$ near $p_0$ in which
 $g$ is written  as

\begin{equation} \label{ggg2}
g=    g^y_{11}dy_1^2+ g^y_{22}dy_2^2 +   g^y_{23} dy_2 \odot dy_3  +  g^y_{33}dy_3^2,
\end{equation}
and $E_1^* =  \sqrt{g^y_{11}}dy_1 $ and $E_2^* \otimes E_2^* +  E_3^* \otimes E_3^* =g^y_{22}dy_2^2 +   g^y_{23} dy_2 \odot dy_3  +  g^y_{33}dy_3^2 $. In this proof, $g_{ij}^{\bullet}$ denotes a function.

 Similarly, there is a coordinate system $z_1,z_2,z_3$ in which
 $g$ is written  as
\begin{equation} \label{ggg3}
g=    g^z_{22}dz_2^2 + g^z_{11}dz_1^2 +   g^z_{13} dz_1 \odot dz_3 +  g^z_{33}dz_3^2,
\end{equation}
and $E_2^* =  \sqrt{g^z_{22}}dz_2 $ and $E_1^* \otimes E_1^* +  E_3^* \otimes E_3^* =g^z_{11}dz_1^2 +   g^z_{13} dz_1 \odot dz_3 +  g^z_{33}dz_3^2$.

There is also a coordinate system $w_1, w_2, w_3$ and
\begin{equation} \label{ggg}
 g= g^w_{33}dw_3^2 + g^w_{11}dw_1^2 +   g^w_{12} dw_1 \odot dw_2 +  g^w_{22}dw_2^2,
\end{equation}
and $E_3^* =  \sqrt{g^w_{33}}dw_3 $ and $E_1^* \otimes E_1^* +  E_2^* \otimes E_2^* = g^w_{11}dw_1^2 +   g^w_{12} dw_1 \odot dw_2 +  g^w_{22}dw_2^2$.

The functions $y_1,z_2, w_3$ form a coordinates system near $p_0$ as $d y_1,d z_2, dw_3$ are linearly independent. We have

 $g= E_1^* \otimes E_1^* + E_1^* \otimes E_1^* +  E_1^* \otimes E_1^*  =g^y_{11}dy_1^2 +  g^z_{22}dz_2^2 +  g^w_{33}dw_3^2 $.

 We denote $y_1$, $z_2$, $w_3$  by $x_1$, $x_2$, $x_3$ respectively and then we can write (\ref{mtr1a}).

\end{proof}

For the metric of the form (\ref{mtr1a}) with  $E_i =\frac{1}{ \sqrt{g_{ii}} } \frac{\partial }{\partial x_i} $, one computes
 \begin{eqnarray} \label{e1f545}
\langle\nabla_{E_j} E_j, E_i\rangle =- \frac{ \frac{\partial g_{jj}}{\partial x_i}}{ 2g_{jj} \sqrt{g_{ii}}}  \ \ {\rm when} \ i \neq j.
\end{eqnarray}

\section{Two distinct Ricci-eigenvalues: $\lambda_1 \neq  \lambda_2= \lambda_3 $.  }

We shall study a $(\lambda,3+m)$-Einstein manifold $(M,g,w)$ with two distinct Ricci-eigenvalues, say  $\lambda_1 \neq  \lambda_2= \lambda_3$.
We can choose an orthonormal Ricci-eigen frame field $\{E_i\}$ in a neighborhood of each point in $\{\nabla w\neq 0\}$  such that $\lambda_1=R(E_1,E_1) \neq  \lambda_2= \lambda_3 $.  When $\nabla w$ is Ricci-eigen, we may choose  such $\{E_i\}$ so that
 $\frac{\nabla w}{ |\nabla w|} =E_1 $ or $\frac{\nabla w}{ |\nabla w|}  \in E_1^{\perp}$. In the latter case we may choose $E_3 = \frac{\nabla w}{ |\nabla w|}$.
  When $\nabla w$ is not Ricci-eigen, we can restrict $\{E_i\}$  as below;

\begin{lemma} Suppose $\lambda_1 \neq  \lambda_2= \lambda_3$.
  If $\nabla w$ is not a Ricci-eigen vector field, then there exists an orthonormal Ricci-eigen frame field $\{E_i\}$ such that $\lambda_1=R(E_1,E_1)$,  $E_1w \neq 0, E_3w \neq 0$ but $E_2w=0$.
\end{lemma}
\begin{proof}
Let $E_1$ be a unit  Ricci-eigen vector field with $\lambda_1=R(E_1,E_1)$.
Set $F_1 := \frac{\nabla w}{ |\nabla w|}$. Let $E_1^{\perp}$ be the 2-d distribution pointwise perpendicular to $E_1$.

 As  $E_1 \neq \pm F_1$, $E_1^{\perp}$
 and  $F_1^{\perp}$ intersect along a one-dimensional subspace. So, we may choose $E_2$ to be in that one-dimensional subspace. Choose
 $E_3$ to be orthogonal to $E_1$ and $E_2$.
  Then $E_2w =g(E_2, \nabla w)=0$. But $E_1 w \neq 0$ and $E_3 w \neq 0$; otherwise, $\nabla w$ will be Ricci-eigen.
\end{proof}
 So, regardless of whether $\nabla w$ is Ricci-eigen or not, we can have $E_2w =0$.
We shall call an orthonormal Ricci-eigen frame field  $\{E_i\}$ to be {\it adapted} if $\lambda_1=R(E_1,E_1) \neq  \lambda_2= \lambda_3 $ and  $E_2w=0$.

\medskip
 For the Codazzi tensor $\mathcal{C}= w^{m+1}   R_{jk}  -\frac{m}{2(m-1)} R w^{m+1}g_{jk} +\frac{1}{ m-1} \lambda  w^{m+1} g_{jk}$,  the Ricci tensor $Rc$ and $\mathcal{C}$ have the same eigenspaces, and the eigen-functions $\mu_i$ of $\mathcal{C}$ satisfy  $\mu_1 \neq  \mu_2= \mu_3 $.

Let $E_{23}$ be the distribution spanned by $E_2$ and $E_3$.
By  Lemma \ref{derdlem} (ii),  $E_{23}$  is integrable.    Then there exists a coordinate neighborhood $(x_1, x_2, x_3)$ so that
\begin{eqnarray} \label{e1f5}
\ \ \  \ \ \ \ \ \ \  g= g_{11} dx_1^2     + g_{22} dx_2^2   +  g_{23} dx_2 dx_3+  g_{33} dx_3^2, \ {\rm with} \ E_1 =\frac{1}{\sqrt{g_{11}}}\frac{\partial}{ \partial x_1},
\end{eqnarray}
 where $g_{ij}$ are functions of $x_1, x_2, x_3$ and  $E_{23}$ is identical to the span of $\frac{\partial}{\partial x_2}$ and $\frac{\partial}{\partial x_3}$.

By  Lemma \ref{derdlem} (iii), the eigen-function $\mu_2$ of $\mathcal{C}$ corresponding to $E_{23}$ is constant along a leaf $L$ of $E_{23}$. So, for $ i=2,3$,
\begin{align} \label{e1f5b}
E_i \mu_2=E_i  \left\{w^{m+1}   \lambda_2  -\frac{m}{2(m-1)} R w^{m+1} +\frac{1}{ m-1} \lambda  w^{m+1} \right \}=0 .
\end{align}

\medskip

As $L$ is umbilic by  Lemma \ref{derdlem} (ii), the second fundamental form $h= \frac{H}{2}g^{\sigma}$, where $H$ is the mean curvature of $L$ and $g^{\sigma}:= g|_{L}$. Henceforth, we shall often denote $ \frac{\partial }{\partial x_i}$ by $ \partial_i$.
Denoting by $\nabla^{\sigma} $ the Levi-Civita connection of $g^{\sigma}$, the Codazzi-Mainardi
equations give;
\begin{align*}
  (\nabla^{\sigma}_{\partial_i} h) (  \partial_j, \partial_k)  - (\nabla^{\sigma}_{\partial_j} h) (  \partial_i, \partial_k)
=   R( \partial_j, \partial_i, \partial_k, E_1 )
\end{align*}
 for $i,j,k \geq 2$. Taking the trace of both sides in $j,k$, we get $\partial_i H = 0$ for $i=2,3$, \cite{CMM}.
We then have
 $ E_i   \langle\nabla_{E_2} E_2, E_1 \rangle  =-\frac{1}{2}E_i H=0,\textrm{  i=2,3}.
$

\medskip
Lemma \ref{derdlem}  (i) gives the following;
for any distinct $i, j \in \{1,2,3\}$ with $\mu_j \neq  \mu_i$,
\begin{align} \label{e1f0c}
  (\mu_j - \mu_i) \langle\nabla_{E_j} E_j, E_i \rangle= E_i \mu_j.
\end{align}

Thus for $i=2,3$,
\begin{align}
0=E_i   \langle\nabla_{E_2} E_2, E_1 \rangle =E_i\left(\frac{E_1\mu_2}{\mu_2-\mu_1}\right)=\frac{E_iE_1\mu_2}{\mu_2-\mu_1}+\frac{(E_1\mu_2)(E_i\mu_1)}{(\mu_2-\mu_1)^2}. \label{noneigenst}
\end{align}

The third equality is by (\ref{e1f5b}). As $\langle\nabla_{E_i}E_1, E_1\rangle=0$,  $\nabla_{E_i}E_1$ is  of the form $a E_2 + bE_3$ and so $\nabla_{E_i}E_1 (\mu_2)=0$.
Then
$E_iE_1\mu_2=(\nabla_{E_i}E_1-\nabla_{E_1}E_i+E_1E_i)\mu_2=-(\nabla_{E_1}E_i)\mu_2$.
We denote $\nabla_{E_i}E_j=:\Gamma_{ij}^kE_k$. Then eq. (\ref{noneigenst}) can be expressed as follows.
\begin{align}
  (E_1\mu_2)\{E_i\mu_1+\Gamma_{11}^i(\mu_2-\mu_1)\}=0 \ \ {\rm for} \  i=2,3. \label{core}
\end{align}

We can prove some technical formulas.
\begin{lemma}\label{elambda}
  Let $(M,g,w)$ be a $(\lambda,3+m)$-Einstein manifold with Ricci-eigenvalues $\lambda_1\neq\lambda_2=\lambda_3$. Consider an adapted frame field $\{E_i\}$ in an open subset $\mathcal{U}$ of $\{\nabla w\neq 0\}$. Then we have the followings:
\begin{align}
 & E_2R=E_2\lambda_1=E_2\lambda_2=0 \label{e2lambda}\\
&E_3\lambda_1=\left\{2\lambda-(m+2)\lambda_1\right\}\frac{E_3w}{w} \label{e3lambda1}\\
&E_3\lambda_2=\left\{\frac{m}{2}\lambda_1+\lambda-(m+1)\lambda_2\right\}\frac{E_3w}{w} \label{e3lambda2}\\
&E_1\lambda_1=(2\lambda-m\lambda_1-2\lambda_2)\frac{E_1w}{w}+H(\lambda_2-\lambda_1) \label{e1lambda1}\\
&E_1\lambda_2=(\lambda-\lambda_2-\frac{m}{2}\lambda_1)\frac{E_1w}{w}-\frac{H}{2}(\lambda_2-\lambda_1) \label{e1lambda2}
\end{align}

\end{lemma}

\begin{proof}  Lemma \ref{derdlem} (i) gives $\langle\nabla_{E_2} E_1, E_3\rangle = - \langle E_1, \nabla_{E_2}E_3\rangle=0$.
From (\ref{hpwdef}),
\begin{align*}
0=  \nabla dw(E_2,E_1)=E_2E_1w-(\nabla_{E_2}E_1)w=E_2E_1w.
\end{align*}
Similarly, we can get $E_2E_3w=0$. Taking $E_2$-derivative to (\ref{consteq}), we get
\begin{align*}
 \frac{w^2}{m}E_2R+2(m-1)\left\{(E_2E_1w)(E_1w)+(E_2E_3w)(E_3w)\right\}=0.
\end{align*}
Thus we get $E_2R=0$.
From (\ref{e1f5b}) and $E_2w=0$,
\begin{align*}
  0=E_2\mu_2=E_2\left\{w^{m+1}\lambda_2+\frac{w^{m+1}}{2(m-1)}(2\lambda-mR)\right\}=w^{m+1}(E_2\lambda_2)
\end{align*}
Thus $E_2\lambda_2=0$ and this implies that $E_2\lambda_1$ is also zero. Applying $E_1$ and $E_3$ to the formula (ii) of Lemma \ref{kklem},
\begin{align}
  E_1R=(4\lambda-2m\lambda_1-4\lambda_2)\frac{E_1w}{w}, \ \ \ \ \ \ \  \label{e3rcompue}\\
  E_3R=(4\lambda-2\lambda_1-2(m+1)\lambda_2)\frac{E_3w}{w}.\label{e3rcompu}
\end{align}
 $E_3\mu_2$ is computed as follows.
\begin{align*}
 0= E_3\mu_2=&(m+1)w^m(E_3w)\lambda_2+w^{m+1}(E_3\lambda_2)+\frac{m+1}{2(m-1)}w^m(E_3w)(2\lambda-mR)\\
&\qquad-\frac{m}{2(m-1)}w^{m+1}(E_3R)\\
=&w^m(E_3w)\left\{(m+1)\lambda_2-\frac{m}{2}\lambda_1-\lambda\right\}+w^{m+1}(E_3\lambda_2)
\end{align*}
The last equality is obtained from (\ref{e3rcompu}). Thus we get (\ref{e3lambda2}) and we can obtain (\ref{e3lambda1}) by comparing (\ref{e3lambda2}) and (\ref{e3rcompu}).
Next, we compute $E_1\mu_2$ in the same manner;
\begin{align*}
  E_1\mu_2=w^m(E_1w)(\lambda_2+\frac{m}{2}\lambda_1-\lambda)+w^{m+1}(E_1\lambda_2).
\end{align*}

From $\frac{E_1\mu_2}{\mu_2-\mu_1}=  \frac{E_1\mu_2}{w^{m+1} (\lambda_2-\lambda_1)}=-\frac{H}{2}$, we get (\ref{e1lambda2}) and
 then (\ref{e1lambda1}) follows from (\ref{e3rcompue}).
\end{proof}

 \begin{lemma} \label{3c}
 Let $(M,g,w)$ be a $(\lambda,3+m)$-Einstein manifold with Ricci-eigenvalues $\lambda_1\neq\lambda_2=\lambda_3$. Suppose  that  $\nabla w$ is not a Ricci-eigen vector field. Consider an adapted frame fields $\{E_i\}$ in an open subset $\mathcal{U}$ of $\{\nabla w\neq 0\}$. Then there exists locally a coordinate system $x_1, x_2, x_3$ in which
 \begin{equation} \label{16s}
g = g_{11}(x_1,  x_3)dx_1^2 +   g_{33}(x_1, x_3) k(x_2, x_3) dx_2^2  + g_{33}(x_1, x_3) dx_3^2,
\end{equation}
 for a function $k$, with $E_i =\frac{1}{ \sqrt{g_{ii}}} \frac{\partial }{\partial x_i} $ for $i=1,2,3$, where $g_{ii}$ is the coefficient function of $dx_i^2$. And for functions $\alpha, \beta$,
\begin{align}
& \nabla_{E_1} E_1=  - \alpha  E_3, \ \ \ \ \ \ \ \ \nabla_{E_1} E_2=0, \ \ \  \ \ \ \ \ \ \nabla_{E_1} E_3= \alpha E_1, \nonumber  \\
&\nabla_{E_2} E_1=\frac{H}{2}E_2,  \  \ \ \ \nabla_{E_2} E_2= -\frac{H}{2}E_1 + \beta E_3, \ \  \nabla_{E_2} E_3=  -\beta E_2,  \label{0ca2}\\
& \nabla_{E_3} E_1 = \frac{H}{2}E_3,  \ \   \  \ \ \   \ \   \    \nabla_{E_3} E_2=0, \ \ \ \ \ \  \ \ \       \nabla_{E_3} E_3=-\frac{H}{2} E_1. \nonumber
 \end{align}
 The mean curvature $H$ does not depend on $x_2, x_3$.
 Ricci curvature components are as follows;
\begin{eqnarray}
\ \ \ \ \ R_{11} := R(E_1, E_1)= -E_1H +2\alpha \beta-\frac{H^2}{2}. \ \ \   \ R_{ij} =0 \  {\rm for} \ i \neq j. \label{rc4} \\
R_{22} =R_{33}= -\frac{E_1H}{2}-\frac{H^2}{2}+\alpha \beta  + E_3\beta    -\beta^2. \ \ \ \ \ \ \ \ \ \  \nonumber
 \end{eqnarray}
Futhermore, the followings hold;
\begin{eqnarray}
\ \ \ \ \   \ E_3\alpha + \alpha^2 + \alpha \beta =0, \   \ \
  E_2 \alpha =0, \ \ \    E_1\beta +\frac{H}{2}(\alpha+ \beta )  =0.  \label{ef2}
 \end{eqnarray}

 \end{lemma}
 \begin{proof}

\bigskip
From $E_2 w=0$ and (\ref{e2lambda}), $E_2 \mu_1 =0$.  So, $\langle\nabla_{E_1} E_1, E_2\rangle=0$ from (\ref{e1f0c}).
  The equation {\rm (\ref{hpwdef})} gives
  \begin{align*}
    0=\nabla d w (E_1, E_2)   = E_1 E_2 w - (\nabla_{E_1} E_2) w
 =- \langle\nabla_{E_1} E_2, E_3\rangle(E_3 w).
  \end{align*}
 As $\nabla w$ is not Ricci-eigen and $E_2 w=0$, we have  $E_1 w \neq 0$ and $E_3w \neq 0$.
 So, $ \langle\nabla_{E_1} E_2, E_3\rangle=0$ and $\nabla_{E_1} E_2=0$. Now, $\nabla_{E_1} E_3= \alpha E_1$ for a function $\alpha$.
 As  a leaf $L$ of $E_{23}$ is umbilic with $h  = \frac{H}{2} g|_{L}$, $\ 0=h(E_2, E_3) = -\langle \nabla_{E_3} E_2, E_1  \rangle $.
Then
$0=\nabla d w (E_3, E_2)   = E_3 E_2 w - (\nabla_{E_3} E_2) w=  - \langle\nabla_{E_3} E_2, E_3\rangle(E_3 w).$
 So, $\langle\nabla_{E_3} E_2, E_3\rangle=0$ and $\nabla_{E_3} E_2=0$. As $\langle\nabla_{E_3} E_3, E_1 \rangle  =-\frac{1}{2} H $,  $\nabla_{E_3} E_1 = \frac{H}{2}E_3$.
Similarly, $\nabla_{E_2} E_1=\frac{H}{2}E_2$.  We can easily compute the rest of formulas in (\ref{0ca2}).

 %  In coordinates we compute  $H(s) g_{ij} = \frac{1}{(g_{11})^{\frac{3}{2}}} \partial_1 g_{ij}$ for $i,j \geq 2$.

As $[E_1, E_2] = -\frac{H}{2}E_2 $  and  $[E_1, E_3] =\alpha E_1 -\frac{H}{2}E_3$, we find that
$E_{12}, E_{13}, E_{23}$ are all integrable.

By Lemma \ref{claim112b8} there exists a new coordinate system, which we still denote by $x_1, x_2, x_3$ in which
\begin{equation}
g = g_{11}(x_1,  x_3)dx_1^2 +    g_{22}(x_1, x_2, x_3) dx_2^2  + g_{33}(x_1,  x_3) dx_3^2,
\end{equation}
with $E_i =\frac{1}{ \sqrt{g_{ii}}} \frac{\partial }{\partial x_i} $;
  here $g_{11}$ and $g_{33}$ does not depend on $x_2$, from $\langle\nabla_{E_1} E_1, E_2\rangle=0$,  $\langle\nabla_{E_3} E_3, E_2\rangle=0$ and (\ref{e1f545}).

   We compute Jacobi identity $[[E_1, E_2], E_3]  +[[E_2, E_3], E_1] + [[E_3, E_1], E_2]=0$ as follows;
$(E_1 \beta) E_2+\frac{H}{2} (\alpha+\beta) E_2   +(E_2 \alpha) E_1=0,$
 so we get $E_2 \alpha =0$,  and $E_1\beta+\frac{H}{2}(\alpha+ \beta ) =0 $.

We compute some curvature components of $g$;
\begin{align*}
  &R_{1221}:= R(E_1, E_2, E_2, E_1)=-\frac{E_1H}{2}+ \alpha \beta -\frac{H^2}{4},\\
  &R_{1331}= -\frac{E_1H}{2}-E_3\alpha - \alpha^2-\frac{H^2}{4},  \ \ \ \ \ \ \ \ \ \  \\
  &R_{2332}=-\frac{H^2}{4}+ E_3\beta    -\beta^2.  \ \ \ \ \ \ \ \ \ \
\end{align*}
Other $R_{ijkl}=0$.
By $R_{22}=R_{33}$, we get $ R_{1221} =R_{1331} $ which is $E_3\alpha + \alpha^2 + \alpha \beta =0  $. From these, we get Ricci curvature formulas.

As $\frac{H}{2}=\langle\nabla_{E_2} E_1, E_2\rangle = \frac{ \partial_1 g_{22}}{ 2g_{22} \sqrt{g_{11}}}$ by (\ref{e1f545}) and
 $\frac{H}{2}=\langle\nabla_{E_3} E_1, E_3\rangle= \frac{ \partial_1 g_{33}}{ 2g_{33} \sqrt{g_{11}}}$, we get $ \frac{\partial_1 g_{22}}{ 2g_{22} \sqrt{g_{11}}}=\frac{ \partial_1 g_{33}}{ 2g_{33} \sqrt{g_{11}}} $. Then
$g_{22}=g_{33} k(x_2, x_3)$ for a function $k$.  From $E_i H =0$ for $i=2,3$, we get $H=H(x_1)$.
 \end{proof}

In the (proofs of) lemmas below,  $c_i$, $i=2,3, \cdots, $ will denote a function, mostly with $x_1$ variable only.

 \begin{lemma} \label{41d3}
Under the hypothesis of Lemma {\rm \ref{3c}}, we further assume that  $\alpha =\langle\nabla_{E_1} E_3, E_1 \rangle= 0$.
 Then there exists locally a coordinate system $(x_1, x_2, x_3)$ in which
\begin{align} \label{gg5}
  g = dx_1^2 +    e^{ \int^{x_1}_{c} H(u_1)du_1} (q( x_3))^2  dx_2^2  + e^{ \int^{x_1}_c H(u_1)du_1} dx_3^2
\end{align}
 for a constant $c$, a function $q(x_3)$, $E_1 =\frac{\partial}{\partial x_1}$, $E_2 =   \frac{1}{q}e^{ -\frac{1}{2}\int^{x_1}_c H(u_1)du_1} \frac{\partial}{\partial x_2}$ and $E_3 =   e^{ -\frac{1}{2}\int^{x_1}_c H(u_1)du_1} \frac{\partial}{\partial x_3}.$  And  $R(E_1, E_1)= -\partial_1H -\frac{H^2}{2}$.

 \end{lemma}

\begin{proof} We start from (\ref{16s}).
As $\alpha=0$,  by (\ref{e1f545}) $\partial_3 g_{11}=0$.
By replacing $x_1$ by a new variable, which we still denote by $x_1$,  we can replace $g_{11}(x_1)dx_1^2$ by $dx_1^2$ and we have
\begin{align*}
  g = dx_1^2 +    g_{33}(x_1, x_3) k(x_2, x_3) dx_2^2  + g_{33}(x_1, x_3) dx_3^2.
\end{align*}
 We remark that when we replace coordinates, it may affect statements proved in previous ccordinates,  but by checking  through the unchangeable frame $\{ E_i\} $, the statements still can be verified in new coordinates.

Now, by definition, $ H(x_1)=2g(E_3, \nabla_{E_3} E_1) =\partial_1\ln g_{33}  $ so that $ g_{33}=h_2(x_3)e^{ \int_c^{x_1} H(u_1) du_1}$  for a function $h_2$ and a constant $c$.
We replace $h_2(x_3) dx_3^2$ by $dx_3^2$, as done above, and write $(r(x_2, x_3))^2 $ for $k(x_2, x_3) h_2(x_3) $ so that
\begin{align} \label{gg1}
  g = dx_1^2 +    e^{ \int H(x_1)} (r(x_2, x_3))^2  dx_2^2  + e^{ \int H(x_1)} dx_3^2
\end{align}
 with $E_1 = \frac{\partial}{\partial x_1}$ and $E_3 =   e^{ -\frac{1}{2}\int H(x_1)} \frac{\partial}{\partial x_3}.$
 From {\rm (\ref{hpwdef})} and $\alpha =0$ we have $0=\nabla d w (E_1, E_3)= E_1 E_3 w $, which gives
   $\partial_1  (e^{ -\int \frac{H(x_1)}{2}} \partial_3w)=0 .$ From this for a function $q$  we get
 \begin{equation} \label{34a}
\partial_3w=    e^{ \int \frac{H(x_1)}{2}}q(x_3).
\end{equation}
From {\rm (\ref{hpwdef})}, we also have $\nabla d w (E_2, E_2)= \nabla d w (E_3, E_3) $ which reduces to $ - \langle \nabla_{E_2} E_2, E_3\rangle E_3 w=E_3E_3 w$. This equation gives
  $\frac{\partial_3 g_{22}}{2 g_{22} \sqrt{g_{33}}}\partial_3 w=\partial_3(\frac{\partial_3w}{\sqrt{g_{33}}}) =  \partial_3q $, where $g_{ii}$ means the coefficient of $dx_i^2$ in (\ref{gg1}). From this we get
$\frac{\partial_3 r}{r}= \frac{\partial_3 q}{q} .  $  Then $r(x_2, x_3)=q (x_3) p_1( x_2)$ for a function  $p_1$.
Replacing  $p_1^2(x_2)dx_2^2$ by  $dx_2^2$,
\begin{align*}
  g = dx_1^2 +    e^{\int H(x_1)} (q(x_3))^2  dx_2^2  + e^{ \int H(x_1)} dx_3^2.
\end{align*}
From (\ref{rc4}) we get $R_{11}$.
\end{proof}

 \begin{lemma} \label{41d}
 Under the hypothesis of Lemma \ref{3c}, we assume that  $\alpha \neq 0$.
 Then there exists locally a coordinate system $(x_1, x_2, x_3)$ in which
\begin{equation} \label{32av}
g =g_{11}(x_1,  x_3) dx_1^2 +     g_{33}(x_1, x_3) v(x_3) dx_2^2  +  g_{33}(x_1, x_3)  dx_3^2,
\end{equation}
for a function $v$, with $E_1 = \frac{\partial_1}{\sqrt{g_{11}}}$, $E_2=\frac{\partial_2}{\sqrt{g_{33}v}}$, $E_3=\frac{\partial_3}{\sqrt{g_{33}}}$ and
\begin{equation} \label{32ajh}
\partial_3 g_{11}=c_{3}(x_1)g_{33}\sqrt{g_{11} \cdot v}.
\end{equation}
for a function $c_3 \neq 0$.
 \end{lemma}

\begin{proof} We again start from (\ref{16s}).
By (\ref{rc4}), (\ref{ef2}) and (\ref{e2lambda}), $0= \partial_2R_{11} = \partial_2 \{-\frac{1}{\sqrt{g_{11}}}\partial_1H\} +2\alpha \partial_2\beta= 2\alpha \partial_2\beta .$
As $\alpha \neq 0$,  $\partial_2 \beta=0.$
From $\beta=\langle\nabla_{E_2} E_2, E_3\rangle $ and (\ref{e1f545}), we get $-2\beta(x_1, x_3)\sqrt{g_{33}}(x_1, x_3)=\partial_3 (\ln g_{22} ) =\partial_3 (\ln g_{33}(x_1, x_3)  +  \ln k(x_2, x_3)).  $

Then $\partial_2 \partial_3 ( \ln k(x_2, x_3)) =0$, so  $k(x_2, x_3)= q(x_2) v(x_3)  $ for functions $v$ and $q$.
Replacing $q(x_2)dx_2^2$ by $dx_2^2$, we may write
$g =g_{11}(x_1,  x_3)dx_1^2 +     g_{33}(x_1, x_3) v(x_3) dx_2^2  +  g_{33}(x_1, x_3)  dx_3^2.$

From (\ref{ef2}) and (\ref{e1f545}),  we get
 $0=2 \sqrt{g_{33}}( \frac{E_3 \alpha}{\alpha} + \alpha +\beta) =  ( \partial_3 \ln  (\alpha^2) + \partial_3 \ln g_{11} -\partial_3 \ln g_{22})  $.
 We get  $\partial_3\left\{ \ln  \frac{(\alpha^2 { g_{11}} )}{ {g_{22}}} \right\} =0$  or equivalently  $\partial_3  \left\{ \frac{(\partial_3 g_{11})^2}{ {g_{11}g_{22} g_{33}}}\right\}  =0.$ So, we get  (\ref{32ajh})
 % \begin{eqnarray} \label{0f21b2}\partial_3 g_{11}=c_{3}(x_1)g_{33}\sqrt{g_{11}p}\end{eqnarray}
 for a function $c_3$.
 As $\alpha \neq 0$, $\partial_3 g_{11} \neq 0 $. So, $c_3 \neq 0$.
\end{proof}

\section{When $\nabla w$ is not a Ricci-eigen field and $E_1 \mu_2  \neq 0$.  }
For a $(\lambda,3+m)$-Einstein manifold $(M,g,w)$, we are assuming that $\lambda_1 \neq  \lambda_2= \lambda_3 $, $\nabla w$ is not a Ricci-eigen field.
Then there are two possible cases by (\ref{core}): for an adapted frame field $\{E_i\}$ of $g$ in the form of (\ref{16s}),
  \begin{itemize}
     \item $ E_3\mu_1-\alpha(\mu_2-\mu_1)=0$
\item $E_1\mu_2=0$
        \end{itemize}
In this section, we consider the first case. To avoid redundant discussion, it is assumed that $E_1\mu_2\neq 0$.
First we show that $\alpha$ must be zero.

\begin{lemma} \label{e1mu2not01}
   Let $(M,g,w)$ be a $(\lambda,3+m)$-Einstein manifold with Ricci-eigen values $\lambda_1\neq\lambda_2=\lambda_3$.  Suppose  that  $\nabla w$ is not a Ricci-eigen vector field and that $E_1\mu_2\neq 0$
for an adapted field $\{E_i\}$ in an open subset $\mathcal{U}$ of $\{\nabla w\neq 0\}$.
Then $\alpha=\langle\nabla_{E_1}E_3,E_1\rangle=0$.
\end{lemma}

\begin{proof}
 We shall show that there is no metric with $\alpha\neq0$.
Suppose $\alpha \neq 0$. Then we may use $g$ in the form of (\ref{32av}).
From {\rm (\ref{hpwdef})}, we also have $\nabla d w (E_2, E_2)= \nabla d w (E_3, E_3) $ which reduces by (\ref{0ca2})   to $ - \beta E_3 w=E_3E_3 w$.
  Thus we have $\frac{\partial_3g_{22}}{2g_{22}}(E_3w)= \partial_3(E_3w)$.
By integrating, we obtain $E_3w=c_4(x_1)\sqrt{g_{22}}$ for a function $c_4$. As $g_{22}=v(x_3)g_{33}$, we get
\begin{align}
  \partial_3w=c_4(x_1)g_{33}\sqrt{v}. \label{partial3w2}
\end{align}

 $E_3\mu_1-\alpha(\mu_2-\mu_1)=0$ gives $\partial_3\mu_1-\frac{\partial_3g_{11}}{2g_{11}}(\mu_2-\mu_1)=0$. Since $\partial_3\mu_2=0$, we get $2\partial_3\{\ln(\mu_2-\mu_1)\}
+\partial_3(\ln g_{11})=0$.
As  we  have  $\partial_2\mu_1=\partial_2\mu_2=0$ from (\ref{e2lambda}), the integration of the above gives
\begin{align}
  (\mu_2-\mu_1)\sqrt{g_{11}}=c_6(x_1) \label{mu2mu1}
\end{align}
for a function $c_6(x_1)\neq 0$. From (\ref{e3lambda1}) and (\ref{e3lambda2}), we have
\begin{align}
 {E_3\{2\lambda-m\lambda_1-2\lambda_2\}}+ (m+1)\frac{E_3w}{w}({2\lambda-m\lambda_1-2\lambda_2})=0.
\end{align}
Then, $E_3[w^{m+1}(2\lambda-m\lambda_1-2\lambda_2)] =0$, so for a function $c_7(x_1)$ we get
\begin{align} \label{356}
  w^{m+1}(2\lambda-m\lambda_1-2\lambda_2)=c_7(x_1).
\end{align}
 But $\lambda_2=\lambda_1+\frac{c_6}{\sqrt{g_{11}}}w^{-m-1}$ by (\ref{mu2mu1}). So we get
\begin{align*}
  w^{m+1}\{2\lambda-(m+2)\lambda_1\}=c_7+\frac{2c_6}{\sqrt{g_{11}}}.
\end{align*}
Now (\ref{e3lambda1}) gives $E_3[w^{m+2}\{ 2\lambda-(m+2)\lambda_1 \}]=0$, so
$w^{m+2}\{ 2\lambda-(m+2)\lambda_1 \}=c_8(x_1)$ for a function $c_8$. Therefore,
\begin{align} \label{350}
  \frac{c_8}{w}=c_7+\frac{2c_6}{\sqrt{g_{11}}}.
\end{align}
Taking $\partial_3$-derivative, we get $\frac{c_8}{w^2}(\partial_3w)=\frac{c_6}{g_{11}\sqrt{g_{11}}}(\partial_3g_{11})$. Apply (\ref{32ajh}) and (\ref{partial3w2}) to this, we get
\begin{align} \label{351}
  \frac{c_4c_8}{w^2}=\frac{c_3c_6}{g_{11}}.
\end{align}
Taking square of (\ref{350}),  $(\frac{c_8}{w}-c_7)^2=\frac{4c_6^2}{{g_{11}}}$. Recall that $c_3 c_6 \neq 0$. Put this into (\ref{351}), we get
$\frac{c_4c_8}{w^2}=\frac{c_3}{4c_6}(\frac{c_8}{w}-c_7)^2 $ so that $c_4c_8\frac{4c_6}{c_3}=(c_8-c_7w)^2 $. As
$c_i$'s in this formula depend only on $x_1$,
 if $c_7\neq 0$, then $w$ is a function of $x_1$ only, which means that $\nabla w$ is a Ricci-eigen field, a contradiction. Thus $c_7=0$, By (\ref{356}) $2\lambda-m\lambda_1-2\lambda_2=0$. But note that $\mu_2=\frac{w^{m+1}}{2(m-1)}(2\lambda-2\lambda_2
-m\lambda_1)$. So, $\mu_2=0$. This is a contradiction to the hypothesis $E_1\mu_2\neq 0$.
\end{proof}

Now we analyze $(g, w)$.

\begin{lemma} \label{e1mu2not0}
   Let $(M,g,w)$ be a $(\lambda,3+m)$-Einstein manifold with Ricci-eigen values $\lambda_1\neq\lambda_2=\lambda_3$.
      Suppose  that  $\nabla w$ is not a Ricci-eigen vector field and that $E_1\mu_2\neq 0$
for an adapted field $\{E_i\}$ in an open subset $\mathcal{U}$ of $\{\nabla w\neq 0\}$.

    Then there exist local coordinates $(x_1,x_2,x_3)$ in which $(g,w)$ can be as follows:
\begin{align}
  g=dx_1^2+(\eta(x_1))^2 (p'(x_3))^2dx_2^2+ (\eta(x_1))^2dx_3^2 \label{result1a}
\end{align}
where $p(x_3)$ is a non-constant positive solution of
\begin{align}
  (p')^2+\rho p^2+\frac{2a}{m-1}p^{1-m}=\frac{\mu}{m-1} \label{koeq}
\end{align}
 for  constants $a\neq0$, $\rho $, $\mu$ and $\eta:=\eta(x_1)$ is a non-constant positive solution of
 \begin{align}\label{eta3}
   (\eta^{'})^2+\frac{\lambda}{m+2}\eta^2=\rho .
\end{align}
And the potential function $w=p(x_3) \eta(x_1)$. Conversely, any metric $g$ and $w$ in the above form  satisfy that
 $\nabla w$ is not a Ricci-eigen vector field,
$\lambda_1\neq\lambda_2=\lambda_3$, $E_1\mu_2\neq 0$ and {\rm {\rm (\ref{hpwdef})}}.

\end{lemma}

\begin{proof} As $E_1\mu_2 \neq 0$,
we have $ E_3\mu_1-\alpha(\mu_2-\mu_1)=0$. As  $\alpha=0$ from Lemma \ref{e1mu2not01}, $E_3\mu_1=0$. Calculating this using Lemma \ref{elambda},
\begin{align*}
  E_3\mu_1=&(m+1)\lambda_1 w^m(E_3w)+w^{m+1}(E_3\lambda_1)-\frac{m}{2(m-1)}w^{m+1}(E_3R)\\
  &+\frac{m+1}{2(m-1)}w^m(E_3w)(2\lambda-mR)\\
=  &\left(\lambda-\frac{m+2}{2}\lambda_1\right)w^m (E_3w)=0.
\end{align*}
Thus we obtain $\lambda_1=\frac{2\lambda}{m+2}$. We may assume $g$ as in (\ref{gg5}).
From the value of $\lambda_1=R_{11}$ in  Lemma \ref{41d3}, $H(x_1)$ satisfies $-\partial_1H-\frac{H^2}{2}=\frac{2\lambda}{m+2}$.
We set $H=2\frac{\partial_1\eta}{\eta}$ for a function $\eta(x_1) $ so that $e^{ \int H(x_1)} = c^2 \cdot \eta^2(x_1)$ for a constant $c>0$.
And $\eta(x_1)$ satisfies
\begin{align} \label{me2g}
\partial_1 \partial_1\eta + \frac{\lambda}{m+2}\eta =0, \ \ {\rm and }  \  (\partial_1\eta)^2+\frac{\lambda}{m+2}\eta^2=\rho  \  {\rm for} \ {\rm a} \ {\rm  constant} \ \rho.
\end{align}
From  (\ref{gg5}), the metric can be written as $ g=dx_1^2+\eta^2(x_1)q^2(x_3)c^2dx_2^2+\eta^2(x_1)c^2 dx_3^2$. Replacing $c^2 dx_2^2 $  by $dx_2^2 $ and  $c^2 dx_3^2 $  by $dx_3^2 $,
we rewrite;
\begin{align} \label{me2}
  g=dx_1^2+\eta^2(x_1)Q^2(x_3)dx_2^2+\eta^2(x_1)dx_3^2.
\end{align}

As $E_1\lambda_1=0$, by (\ref{e1lambda1}),
\begin{align*}
  0=&E_1\lambda_1=\left(2\lambda-m\lambda_1-2\lambda_2\right)\frac{E_1w}{w}+H(\lambda_2-\lambda_1)\\
  =&\left(2\frac{E_1w}{w}-H\right)\left(\frac{2\lambda}{m+2}-\lambda_2\right).
\end{align*}
Since we assume $\lambda_1 \neq\lambda_2$, we have
\begin{align*}
  \frac{\partial_1 \eta}{\eta}=\frac{H}{2}=\frac{E_1w}{w}=\frac{\partial_1w}{w}.
\end{align*}
Thus $w=\eta(x_1)p( x_3)$ for a function $p$. Neither  $\eta$ nor $p$ can be constant; otherwise $\nabla w$ becomes a Ricci-eigen field. From $\nabla dw(E_2,E_2)=\nabla dw(E_3,E_3)$ and (\ref{0ca2}), we again get $E_3E_3w=-\beta(E_3w)$. We denote $\partial_3X$ by $X'$.
Since $E_3 = \frac{\partial_3}{\eta} $, we get
$\beta=-\frac{p^{''}}{p^{'}\eta}  $.

From (\ref{me2}),  $\beta= g(\nabla_{E_2} E_2, E_3)=-\frac{Q'}{Q\eta}$, so  we get $Q=Cp'$ for a constant $C$. Assigning the values obtained so far to $\nabla dw(E_3,E_3)=\frac{w}{m}(\lambda_2-\lambda)$, (\ref{rc4}) and (\ref{me2g}) give
\begin{align*}
  \frac{p''}{\eta}+\frac{\partial_1 \eta}{\eta}(p\partial_1\eta)=\frac{p\eta}{m}  \left\{\frac{\lambda}{m+2}-(\frac{\partial_1\eta}{\eta})^2-\frac{1}{\eta^2}\frac{p'''}{p'}-\lambda\right\},
\end{align*}
which reduces to
\begin{align} \label{pppp}
   \frac{m+1}{m}\rho =-\frac{p''}{p}-\frac{1}{m}\frac{p'''}{p'}.
\end{align}
 Multiplying $2mpp'$ and splitting $pp''$-terms,
\begin{align*}
2\rho  (m+1)pp'=-2p'p''-2p p'''-2(m-1)p'p''.
\end{align*}
Integrating, for a constant $\mu_0$ we get
\begin{align} \label{pppp2}
   \mu_o=\rho  (m+1) p^2+2p p''+(m-1)(p')^2.
\end{align}
Meanwhile, the constant $\mu$ of Lemma \ref{kklem} is
\begin{align*}
  \mu=&\frac{w^2}{m}\{R+(m-3)\lambda\}+(m-1)|\nabla w|^2\\
=&\frac{p^2\eta^2}{m}\left\{\left(\frac{4\lambda}{m+2}-2\left(\frac{\partial_1\eta}{\eta}\right)^2-\frac{2}{\eta^2}\frac{p'''}{p'}\right) +(m-3)\lambda\right\}  \\& +(m-1)\{p^2(\partial_1\eta)^2+(p')^2\}\\
  =&\frac{(m-2)(m+1)}{m}\rho p^2-\frac{2}{m}p^2\frac{p'''}{p'}+(m-1)(p')^2.
\end{align*}
The last equality holds by $(\partial_1\eta)^2+\frac{\lambda}{m+2}\eta^2=\rho $. Put (\ref{pppp}) into the right hand side of the above
equation, we can see that $\mu=\mu_0$.

Multiplying $p^{m-2}p'$ to (\ref{pppp2}) and integrating, we get
\begin{align}
  \frac{\mu}{m-1}p^{m-1}=\rho p^{m+1}+p^{m-1}(p')^2+\frac{2a}{m-1} \label{odep'}
\end{align}
for a constant $a$. We  obtain (\ref{koeq}). Differentiating we get
\begin{align}
  p''+\rho p=ap^{-m}. \label{koeqdif}
\end{align}

Now compute $\mu_2$ using (\ref{me2g}) and (\ref{koeqdif});
\begin{align*}
  \mu_2=&w^{m+1}\left\{\lambda_2+\frac{\lambda}{m-1}-\frac{m}{2(m-1)}R\right\}\\
=&\frac{p^{m+1}\eta^{m+1}}{m-1}\left\{\frac{\lambda}{m+2}+\left(\frac{\partial_1\eta}{\eta}\right)^2+\frac{1}{\eta^2}\frac{p'''}{p'}\right\}
=\frac{m}{1-m}a\eta^{m-1}.
\end{align*}
As  $E_1\mu_2\neq0$, $a$ cannot be zero.

 \medskip
We can easily check the converse part.
\end{proof}

Lemma \ref{e1mu2not0} brings us to the Kobayashi's solutions to the ordinary differential equation  \cite[Equation (2.2)]{Ko}, which is identical to (\ref{koeq}) if we match $n \leftrightarrow m+1$, $k \leftrightarrow \frac{\mu}{m-1}$ and $R \leftrightarrow m(m+1)\rho $.
As $\eta >0$ and  $w = \eta p$ are positive in the open dense subset of the manifold under consideration, so is $p$. Therefore in solving (\ref{koeq}), we may assume a positive initial value for $p$ just as in Kobayashi's paper.
In his study of equation, Kobayashi discussed complete conformally flat static spaces but did not mention on the completeness of metrics of the form $p'(x_3)^2dx_2^2+dx_3^2$ that we have in (\ref{result1a}).
We want to sort out
   possible complete  $(\lambda,3+m)$-Einstein manifolds under the hypothesis of Lemma \ref{e1mu2not0}.   First we state;

\begin{lemma} \label{420}
  If either  $\lambda \geq 0$, or  $\rho >0$ and $\lambda < 0$, then $(g,w)$ of  Lemma {\rm \ref{e1mu2not0}} cannot be complete.
\end{lemma}

\begin{proof}
  The metric $g$ in (\ref{result1a}) can be viewed as a warped product metric $g=dx_1^2+\eta^2\tilde{g}$ where $\tilde{g}=(p'(x_3))^2dx_2^2+dx_3^2$. If either  $\lambda \geq 0$, or  $\rho >0$ and $\lambda < 0$,
from (\ref{me2g}) there exists $0\leq t_0<\infty$ such
  that $\eta(t_0)=0$. Then $g$ is smooth at $t_0$ only if $\tilde{g}$ has positive constant curvature. Differentiating (\ref{koeqdif}), the sectional curvature $k_{\tilde{g}}$ of $\tilde{g}$ satisfies
  \begin{align*}
    k_{\tilde{g}}=-\frac{p'''}{p'}=ma\cdot p^{-m-1}+\frac{\lambda}{m+2}.
  \end{align*}
So, for $\tilde{g}$ to have constant curvature, $a$ must be zero, which is a contradiction.

\end{proof}

For his equation (2.2) Kobayashi made a complete list of conditions for parameters and initial values which comprise (I) through (VI); see page 670 of \cite{Ko}.
 To solve (\ref{koeq}), we shall use that list under the match $n \leftrightarrow m+1$, $k \leftrightarrow \frac{\mu}{m-1}$ and $R \leftrightarrow m(m+1)\rho $.

In our case $a$ cannot be zero and $p$ cannot be a constant. And $\rho  \leq0$  and $\lambda < 0$ from Lemma \ref{420}.   In his case (I), $p(x_3)$ becomes  constant by his Proposition 2.2. In (III.1)$\sim$(III.3) cases;

\begin{lemma} \label{LL3k}
  If one of the conditions {\rm (III.1)$\sim$(III.3)} is satisfied, then $(g,w)$ of  Lemma {\rm \ref{e1mu2not0}} cannot be complete.
\end{lemma}
\begin{proof}
  By proposition 2.4 in \cite{Ko}, under one of (III.1)$\sim$(III.3),  $\lim_{s\rightarrow s_0}(p')^2=\infty$ for some $s_0$. By (\ref{koeq}) either $\lim_{s\rightarrow s_0}p(s)=0$ or $\lim_{s\rightarrow s_0}p(s)=\infty$.
 Then either the sectional curvature $R_{2332}=-\left(\frac{\partial_1\eta}{\eta}\right)^2
+\frac{1}{\eta^2}\left(ma \cdot p^{-m-1}+\frac{\lambda}{m+2}\right)$ diverges or the potential function
$w=p\eta$ diverges at a point of finite distance from a fixed point.
So, we cannot get a complete metric.
\end{proof}

 So,  we have a shorter list of conditions for possibly complete solutions:
\begin{eqnarray}
   \textrm{ IV.1.}&\quad \textrm{ $a>0$, $\rho  \leq0$.} \ \ \ \ \ \ \ \ \ \ \ \ \ \ \ \ \label{tw11}\\
  \textrm{ IV.3.} &\quad \textrm{ $a<0$, $\rho  <0$, $\mu\leq \kappa_0$, $p(0)>\rho_0$.} \nonumber
\end{eqnarray}
Here, when  $a\rho >0$, $\rho_0=\left\{\frac{a}{ \rho }\right\}^{1/(m+1)}$ and $\kappa_0=(m+1)\rho \rho_0^2$.

The next two lemmas follow the line of argument made in Kobayashi's proposition 2.5.
\begin{lemma} \label{26}
 Assume that {\rm (IV.1)} is satisfied and $\lambda < 0$.
 Then a positive solution $p$ of  {\rm (\ref{koeq})} is defined on $\mathbb{R}$ and  $p^{'}(x_3)$ has a unique root.
 And the space $(g,w)$ of  Lemma {\rm \ref{e1mu2not0}}   is defined on $\mathbb{R}^3$ and complete.

\end{lemma}
\begin{proof}
We may denote by $s:=x_3$.
Set $F(p) = \rho p^2+\frac{2a}{m-1}p^{1-m}$. $F(p)$ is decreasing on $p >0$ and $ \lim_{p \rightarrow 0^{+}} F(p) = +\infty$.
As $\frac{\mu}{m-1}=(p')^2+\rho p^2+\frac{2a}{m-1}p^{1-m} \geq  F(p),$ so  $p > p_1$ for some $p_1>0$.

   The equation $h^{''}+\rho h=  \frac{a}{p_1^{m}}$ has a solution $h(s)= c_1 e^{\sqrt{-\rho } s} + c_2 e^{-\sqrt{-\rho } s} +\frac{a}{p_1^{m}\rho }$ when $\rho<0$, and $h(s)= c_1 s + c_2  +\frac{as^2}{2p_1^{m}} $ when $\rho=0$.
$h$ is defined on  $\mathbb{R}$.
Then from (\ref{koeqdif}),
$(p-h)^{''} +\rho (p-h) =         \frac{a}{p^{m}}-  \frac{a}{p_1^{m}}= a\frac{1}{p^{m}p_1^{m}}( p_1^{m}- p^{m})   \leq 0 $. From $h^{'}(0)=p^{'}(0) $ and  $h(0) = p(0)$, we have $p \leq h$. As $p \geq p_1>0$,
we have $p$ defined on  $\mathbb{R}$.

Next, we can show $p^{'}(s)$ has a unique root. The following simple argument  is  from  \cite[Lemma 3.4]{MT}.
We have $p^{''} \geq \frac{a}{p^m}$ from (\ref{koeqdif}). Assume $p^{'} <0$ everywhere, then $p(s) \leq  p(0)$  for all $s \geq  0$. So  $p^{''} \geq C$ for some positive constant $C$ for $s > 0$. This implies $p^{'}(s) > 0$ somewhere,  a contradiction. Similarly, it is impossible to have $p^{'}(s) > 0$ everywhere. Hence $p^{'}(s) = 0$ for some $s$. Since  $p^{''} >0$, the root of $p^{'}(s)=0$ is unique. We may assume $p^{'}(0)=0$. Then
$p$ is an even function. So, $p^{'}(s)$ is a smooth odd function and $p^{''}(0)>0$.
We regard $(x_3,x_2)$ as polar coordinates  $(r, \theta)$ on $\mathbb{R}^2$ so that $dx_3^2 + (p^{'}(x_3))^2 d x_2^2$ can be a smooth metric near $x_3=0$ where $x_2 \in [0, \frac{2 \pi}{ p^{''}(0) }]$.
As $\rho \leq 0$ and $ \lambda<0$, $\eta$ can be some cosine hyperbolic function ($\rho<0$) or an exponential function  ($\rho=0$) from (\ref{eta3}), defined on $\mathbb{R}$.
The metric $g=dx_1^2+\eta^2(p')^2dx_2^2+\eta^2dx_3^2$ on $\{(x_1,x_2,x_3) \ | \ x_1 \in \mathbb{R}, \  0<x_3<\infty, \ x_2 \in [0, \frac{2 \pi}{ p^{''}(0) }] \}$ can be extended smoothly over the set $\{(x_1,x_2,x_3) \ | \ x_1 \in \mathbb{R}, \  0\leq x_3<\infty, \ x_2 \in [0, \frac{2 \pi}{ p^{''}(0) }] \}$ which is  $\mathbb{R}^3$. The potential function
$w=p(x_3) \eta(x_1)$ is also smooth positive on  $\mathbb{R}^3$.
\end{proof}

\begin{lemma} \label{comple22}
Suppose $a<0$, $\rho <0$, $p(0)>\rho_0$ and $\lambda < 0$.

 If $\mu=\kappa_0$, then
  a positive solution $p$ of {\rm (\ref{koeq})} is defined on $\mathbb{R}$ and  $p^{'}(s)$ dose not have a root. And  the space $(g,w)$ of  Lemma {\rm \ref{e1mu2not0}}   is defined on $\mathbb{R}^3$ and complete.

If $\mu<\kappa_0$,  then
  a positive solution $p$ of {\rm (\ref{koeq})} is defined on $\mathbb{R}$ and  $p^{'}(s)$ has a unique root. And  the space $(g,w)$ of  Lemma {\rm \ref{e1mu2not0}}   is defined on $\mathbb{R}^3$ and complete.

\end{lemma}
\begin{proof}
Set $G(p)= \frac{\mu}{m-1} - \frac{2a }{m-1}p^{1-m} -\rho  p^{2}$, which  is increasing on $p > \rho_0 $,  decreasing on  $0<p <\rho_0 $ and  $G( \rho_0) = \frac{\mu}{m-1}- \frac{ k_0 }{m-1} \leq 0$.
We have  $p(0) > \rho_0$.  In case $\mu<\kappa_0$, $G( \rho_0) <0$. And $(p^{'})^2=G(p) \geq 0$ means that $p(s)> \rho_0$. In case $\mu=\kappa_0$, if $p(s_0)=\rho_0$ for some $s_0$  then $p'(s_0)=0$.
From (\ref{koeqdif}) we see that $p(s) \equiv \rho_0$, a contradiction. We again have  $p(s)> \rho_0>0$.

As in the second paragraph in the proof of Lemma \ref{26} we can see that $p$ is defined on  $\mathbb{R}$ and $p(s)>\rho_0$.

Suppose $\mu=\kappa_0$. If $p'(s_1)=0$ for some $s_1$, then
 $G(p(s_1))=0= G( \rho_0) $.
By the behaviour of $G$, $p(s_1)=\rho_0$, which is impossible. So, $p'(s)$ does not have a root.
 As $\rho < 0, \lambda<0$, $\eta$ can be some cosine hyperbolic function from (\ref{eta3}), defined on $\mathbb{R}$.
The metric $g=dx_1^2+\eta^2(p')^2dx_2^2+\eta^2dx_3^2$ and $w=p(x_3) \eta(x_1)$
 are complete on  $\mathbb{R}^3= \{ (x_1, x_2, x_3) \ |  \ x_i \in \mathbb{R}   \}$.

\medskip
Now suppose $\mu<\kappa_0$. Set $ B(p)= - \frac{2a }{m-1}p^{1-m} -\rho  p^{2}$.   As $(p'(s))^2=B(p(s))+\frac{\mu}{m-1}$, we have $B(p(s))\geq -\frac{\mu}{m-1}>-\frac{\kappa_0}{m-1}=B(\rho_0)$. $B(p)$ is increasing on $[\rho_0,\infty)$. Set $b= B|_{[\rho_0,\infty)}$. As $p(s) > \rho_0$,
$p(s)\geq b^{-1}(-\frac{\mu}{m-1})>\rho_0$. In the formula $p''=ap^{-m}-\rho p$, the right hand side is a positive increasing function of $p$ when $p\geq\rho_0$.
So, $p''(s)\geq a\{b^{-1}(-\frac{\mu}{m-1})\}^{-m}-\rho b^{-1}(-\frac{\mu}{m-1}):=\tau_0>a\rho_0^{-m}-\rho \rho_0=0$.

Now, following the argument in the proof of Lemma \ref{26}, one can show that $p^{'}(s)=0$ has a unique root, say $0$,  and that  regarding $(x_3,x_2)$ as polar coordinates $(r,\theta)$ on $\mathbb{R}^2$, the metric $g=dx_1^2+\eta^2(p')^2dx_2^2+\eta^2dx_3^2$ can be defined on $\{(x_1,x_2,x_3) \ | \ x_1 \in \mathbb{R},  \ 0\leq x_3<\infty, \ x_2 \in [0, \frac{2 \pi}{ p^{''}(0) }] \}$ which is  $\mathbb{R}^3$.
$w=p(x_3) \eta(x_1)$ is also smooth positive on  $\mathbb{R}^3$.

\end{proof}

\begin{prop} \label{pr47}
  Let $(M,g,w)$ be a complete $(\lambda,3+m)$-Einstein manifold such that there are exactly two distinct Ricci-eigen values on an open dense subset and that $\nabla w$ is not a Ricci-eigen vector field and $E_1 \mu_2\neq 0$ for
an adapted frame field $\{ E_i\}$.

Then in some coordinates $(x_1,x_2,x_3)$,  $(g,w)$ can be as follows:
\begin{align}
  g=dx_1^2+(\eta(x_1))^2(p'(x_3))^2dx_2^2+(\eta(x_1))^2dx_3^2 \label{result1}
\end{align}
where $p(x_3)$ is a non-constant positive solution on $\mathbb{R}$ of
\begin{align}
  (p')^2+\rho p^2+\frac{2a}{m-1}p^{1-m}=\frac{\mu}{m-1}, \label{koeq2}
\end{align}
\noindent and $\eta(x_1)$  is a non-constant positive solution on $\mathbb{R}$ of
 \begin{align} \label{koeq2v}(\eta^{'})^2+\frac{\lambda}{m+2}\eta^2=\rho ,
\end{align}
for constants $a\neq0$, $\rho $ and $\mu$ under one of the following conditions;

 {\rm (i-1)}  $\lambda<0$, $a>0$, $\rho \leq 0$,

 {\rm (i-2)} $\lambda<0$, $a<0$, $\rho <0$, $\mu\leq (m+1)\rho (\frac{a}{ \rho })^{2/(m+1)}$, $p(0)>(\frac{a}{ \rho })^{1/(m+1)}$.

\medskip
\noindent  And $w=p \cdot \eta$.

Conversely, any $(g, w)$ as in the above is complete and satisfies $\lambda_1\neq\lambda_2=\lambda_3$, $E_1\mu_2\neq 0$ and {\rm (\ref{hpwdef})}.

\end{prop}

\section{When $\nabla w$ is not a Ricci-eigen field and $E_1  \mu_2  = 0$.  }

 In this section we treat the case
when  $E_1\mu_2= 0$, as planned in the beginning of Section 4.  Then $ \mu_2$ is constant from (\ref{e1f5b}).
We can have more information than Lemma \ref{3c}.
 From (\ref{e1f0c}) and $E_1 \mu_2   = 0$, for $j=2,3$ we get $\langle \nabla_{E_j} E_j, E_1 \rangle =0$ and $H=0$ in (\ref{0ca2}). Since a leaf $L$ tangent to $E_{23}$ is umbilic, $L$ is totally geodesic.

 We shall show that the subcase of $\alpha= \langle \nabla_{E_1} E_3, E_1 \rangle=0$ cannot happen.

 \begin{lemma}  \label{LL61z}  Let $(M,g,w)$ be a $(\lambda,3+m)$-Einstein manifold with Ricci-eigen values $\lambda_1\neq\lambda_2=\lambda_3$.
     Suppose  that  $\nabla w$ is not a Ricci-eigen vector field and that $E_1\mu_2= 0$
for an adapted field $E_i$ in an open subset $\mathcal{U}$ of $\{\nabla w\neq 0\}$.
Then there is no solution with $\alpha=0$.
 \end{lemma}

\begin{proof}
Assume $\alpha=0$.
We use $\langle \nabla_{E_j} E_j, E_1 \rangle =0$, $j=2,3,$ and (\ref{e1f545}) for the metric of  Lemma \ref{41d3} and can write
\begin{equation} \label{ggg221}
g= dx_1^2 +     q^2(x_3)dx_2^2+ dx_3^2,
 \end{equation}

Use Lemma \ref{3c} to evaluate the equation {\rm (\ref{hpwdef})} on $(E_i, E_i)$, $i=1,2$ and on $(E_1, E_3)$;
\begin{eqnarray}
 \partial_1 \partial_1 w =- \frac{w}{m}\lambda, \hspace{1.5cm} \label{bb1} \\
\frac{q^{'}}{q} \partial_3 w = -\frac{w}{m}(\frac{q^{''}}{q} +\lambda), \label{bb2}\\
\partial_1 \partial_3 w = 0 \hspace{2cm} \label{bb3}.
 \end{eqnarray}
 (\ref{bb3}) gives $w = w_1(x_1) +  w_3(x_3)$. (\ref{bb1}) then gives $\partial_1 \partial_1 w_1 =-\frac{ w_1(x_1) +  w_3(x_3)}{m}\lambda$. If $\lambda \neq 0$, then $w_3(x_3)$ is a constant. Then $\nabla w$ is Ricci-eigen vector field, a contradiction to the hypothesis.

  %(\ref{bb2}) gives $\frac{p^{''}}{p} +\lambda=0$ so that  the 2-dimensional metric $p^2(x_3) dx_2^2   +  dx_3^2$ has constant sectional curvature $-\frac{p^{''}}{p}=\lambda$.(\ref{bb1}) gives  $\partial_1 \partial_1 (   w(x_1)) = -\frac{ w}{m}\lambda$. So, $w= ...$

If $\lambda = 0$, then $\partial_1 \partial_1 w_1 =0$. So, $w_1 = ax_1 + b$ for constants $a, b$. Note that $q^{''} \neq 0$ because $g$ will have zero curvature if $q^{''} = 0$.
  (\ref{bb2}) gives $m \frac{q^{'}\partial_3 w_3}{q^{''}}+w_1(x_1)+ w_3 =0$, where only $w_1$ depends on $x_1$, so $w_1 =b$. Then $\nabla w$ is Ricci-eigen vector field,  a contradiction to the hypothesis.
\end{proof}

 \begin{lemma}  \label{LL61}
 Let $(M,g,w)$ be a $(\lambda,3+m)$-Einstein manifold with Ricci-eigen values $\lambda_1\neq\lambda_2=\lambda_3$.     Suppose  that  $\nabla w$ is not a Ricci-eigen vector field and $E_1\mu_2= 0$ for an adapted field $E_i$ in an open subset $\mathcal{U}$ of $\{\nabla w\neq 0\}$.

Then
there exist local coordinates $(x_1,x_2,x_3)$ near any point in  $\mathcal{U}$  where $(g,w)$ can be follows:
 \begin{align*}
  g= (p(x_3))^2 dx_1^2 +     (p^{'})^2 dx_2^2+ dx_3^2,
\end{align*}
for a non-constant positive function $p(x_3)$ satisfying
\begin{equation} \label{mu2const1}
(p')^2+\frac{\lambda}{m+2}p^2+\frac{2a}{m}p^{-m}=k.
\end{equation}
for  constants $a \neq 0, k$.
And  $w = {\tau}(x_1) p(x_3)$, where ${\tau}(x_1)$ is  a non-constant  positive function satisfying
$  ({\tau}')^2=-k \cdot {\tau}^2+\frac{\mu}{m-1}$. Conversely, any $(g, w)$ in the above form  satisfies that
 $\nabla w$ is not a Ricci-eigen vector field,
$\lambda_1\neq\lambda_2=\lambda_3$, $E_1\mu_2= 0$ and {\rm {\rm (\ref{hpwdef})}}.

 \end{lemma}

\begin{proof}By Lemma \ref{LL61z}, we may assume $\alpha \neq 0$. For the metric in Lemma \ref{41d}, $H=\frac{\partial_1g_{33}}{g_{33}\sqrt{g_{11}}}=0$, so $\partial_1g_{33}=0$. Replacing $g_{33}(x_3)dx_3^2$ by $dx_3^2$, the metric can be written as
\begin{equation} \label{32a}
g= g_{11}(x_1, x_3) dx_1^2 +     g_{22}( x_3)dx_2^2+ dx_3^2,
 \end{equation}
with $E_i =\frac{1}{\sqrt{g_{ii}}} \frac{\partial }{\partial x_i} $ for $i=1,2,3$.

 (\ref{ef2}) gives $ \partial_3 \alpha  +\alpha \beta +  \alpha^2=0$. As  $\alpha=\langle \nabla_{E_1}E_3,E_1\rangle = \partial_3  \ln (\sqrt{g_{11}})\neq 0$ and $\beta= -  \partial_3 \ln (\sqrt{g_{22}})$,
$  \frac{\partial_3 \partial_3  \ln (\sqrt{g_{11}})}{\partial_3  \ln (\sqrt{g_{11}})}    -  \partial_3 \ln (\sqrt{g_{22}}) +  \partial_3  \ln (\sqrt{g_{11}})=0$.
 Integrating, $ \partial_3 \sqrt{g_{11}} = c_9(x_1) \sqrt{g_{22}} $ and  $ \sqrt{g_{11}} = c_9(x_1) (\int \sqrt{g_{22}}(x_3) dx_3  + c_{10}(x_1))$ for functions $c_9$ and $c_{10}$.

Replacing $c_9^2(x_1)dx_1^2$ by $dx_1^2$ and  denoting $\int  \sqrt{g_{22}} dx_3$ as $p(x_3)$,
 \begin{equation} \label{ggg3b}
g= \{p(x_3)+ c_{10}(x_1)\}^2dx_1^2 +     {p^{'}(x_3)}^2 dx_2^2+ dx_3^2.
\end{equation}
The curvature components are computed from (\ref{rc4});
\begin{align*}
 & R_{11}=R(E_1, E_1)=-2\frac{p^{''}}{p+c_{10}},\quad R_{22}=R_{33}=-\frac{p^{'''}}{p^{'}}-\frac{p^{''}}{p+c_{10}}\\
 & R=-2\frac{p^{'''}}{p^{'}}-\frac{4p^{''}}{p+c_{10}}
\end{align*}

Use $\alpha=\frac{ p^{'}}{ p+ c_{10}}$, $\beta= - \frac{ p^{''}}{ p^{'}}$ and (\ref{0ca2}) to evaluate the equation {\rm (\ref{hpwdef})} on $(E_i, E_i)$, $i=1,2,3$ and on $(E_1, E_3)$, we obtain
\begin{align}
&\frac{\partial_1}{p+ c_{10}} \left(   \frac{\partial_1 w}{p+ c_{10}}\right)  + \frac{ p^{'}}{ p+ c_{10}} (\partial_3 w)   =   \frac{w}{m} \left(- 2\frac{ p^{''}}{ p+ c_{10}} -\lambda\right)  , \label{m01}\\
&\frac{ p^{''}}{ p^{'}} (\partial_3 w) = \frac{w}{m} \left(-\frac{ p^{'''}}{ p^{'}}- \frac{ p^{''}}{ p+ c_{10}} -\lambda\right)\hspace{0.7cm} \label{m02} \\
&\partial_3 \partial_3 w  =\frac{w}{m} \left(-\frac{ p^{'''}}{ p^{'}}- \frac{ p^{''}}{ p+ c_{10}}-\lambda\right) . \hspace{0.4cm}
\label{m03} \\
& \frac{1 }{p+c_{10}}\partial_1 \partial_3 w - \frac{p^{'}}{(p+c_{10})^2}(\partial_1  w) =0. \hspace{1.1cm}  \label{m04}
\end{align}

(\ref{m02}) and  (\ref{m03}) gives $\partial_3 \partial_3 w  = \frac{ p^{''}}{ p^{'}} \partial_3 w $ so that  $\partial_3 w = {\tau}(x_1) p^{'}(x_3)$ and  $w = {\tau}(x_1)\{ p(x_3) + c_{12}(x_1)\}$ for a function $\tau \neq 0$.
Then (\ref{m03})  yields
\begin{equation} \label{et01}
\frac{mp''}{p+c_{12}}=-\frac{p'''}{p'}-\frac{p''}{p+c_{10}}-\lambda.
 \end{equation}
 Take $\partial_1$ to get $-\frac{m p''(\partial_1c_{12})}{(p+c_{12})^2}=\frac{p''(\partial_1c_{10})}{(p+c_{10})^2}$.
Note that $p^{''} \neq 0$; otherwise $R_{11}=R_{22}=R_{33}=0$. Then
\begin{align}
  -m(\partial_1 c_{12} )(p^2 + 2p c_{10} + c_{10}^2) =  (\partial_1 c_{10})( p^2 + 2p c_{12} + c_{12}^2 ). \label{c10c12}
\end{align}
Take $\partial_3$, divide by $p^{'}$ and take  $\partial_3$ again to get $-m(\partial_1 c_{12} ) =  (\partial_1 c_{10})$. If $\partial_1 c_{12} \neq 0$, then $c_{12}=c_{10}$ by (\ref{c10c12}).
This will give a contradiction.
So, $c_{12}$ and $c_{10}$ are both constants.

Put $w = {\tau}(x_1)\{ p(x_3) + c_{12}\}$ into (\ref{m04}),  we get $ (\partial_1{\tau})( c_{10} -c_{12})=0 $.
 Here $  \partial_1 {\tau} \neq 0$; otherwise, $\partial_1w=0$,  a contradiction. So, $c_{10}=c_{12}$.
 Then we can replace the function $p(x_3)+c_{10}$ by another function, still denoted by $p(x_3)$, so that the metric becomes;
 \begin{equation} \label{ggg3b0}
g= \{p(x_3)\}^2dx_1^2 +     \{p^{'}(x_3)\}^2 dx_2^2+ dx_3^2,
\end{equation}
 with $w = {\tau}(x_1) p(x_3)$.
And we may put $c_{10}= c_{12} =0$ in
(\ref{m01})$\sim$(\ref{m04}).

From (\ref{m02}) we get $\frac{ p^{'''}}{ p^{'}}+  (m+1) \frac{p^{''}}{p}  +\lambda =0.$
Integrating this,
\begin{equation} \label{et07u}
 m(p^{'})^2  +    2  p p^{''}+\lambda p^2=km
\end{equation}
for a constant $k$. Multiply by $p^{m-1} p^{'}$ to get
$ mp^{m-1}(p')^3+2p^mp'p''+\lambda p^{m+1}p'=kmp^{m-1}p'.$ Integrate to get
$p^m(p')^2+\frac{\lambda}{m+2}p^{m+2}=kp^m-\frac{2a}{m}$ for a constant $a$. Thus we obtain (\ref{mu2const1}).
Differentiate (\ref{mu2const1}) to get
\begin{equation} \label{et07w}
 p^{''}  +   \frac{\lambda}{m+2} p   =ap^{-m-1}.
\end{equation}
And from (\ref{m01}), we get $m\frac{\partial_1\partial_1{\tau}}{{\tau}}=-2pp^{''}-\lambda p^2-m(p')^2$. By (\ref{et07u}), $\partial_1\partial_1{\tau}=-k{\tau}$. Integrating
\begin{align}
  ({\tau}')^2=-k({\tau})^2+\frac{\mu}{m-1}
\end{align}
for a constant $\mu$. Note that ${\tau}$ is non-constant; otherwise $\nabla w$ becomes Ricci-eigen.
It is easy to see that this $\mu$ is the same as the constant $\mu$ of Lemma \ref{kklem}.

 If $a=0$, then by (\ref{et07w}) $g$ is an Einstein metric with $\lambda_1=\lambda_2=\frac{2\lambda}{m+2}$. So $a$ is not zero.
\end{proof}

Note that (\ref{mu2const1}) and (\ref{et07w}) are also identical to the Kobayashi's equations (2.2) and (2.1) in \cite{Ko}, if we match  $n \leftrightarrow m+2$, $ \ \frac{R}{n-1} \leftrightarrow  \lambda$.

 \medskip
As in Section 4, we look for conditions which can produce complete  $(\lambda,3+m)$-Einstein manifolds under the hypothesis of Lemma \ref{LL61}. Among Kobayashi's list, the conditions I.1 and I.2 gives only constant solutions for $p$.
The metric $g$ of  Lemma \ref{LL61} can be shown to be incomplete under his (III.1)$\sim$(III.3) conditions, as in Lemma \ref{LL3k}.

Since $a \neq 0$, we have the following list of three possible conditions:

\begin{align*}
  \textrm{ IV.1.}&\quad \textrm{ $a>0$, $\lambda \leq0$,}\\
  \textrm{ IV.3.} &\quad \textrm{ $a<0$, $\lambda <0$, $k\leq \kappa_0$, $p(0)>\rho_0$,}\\
  \textrm{ V.} & \quad \textrm{ $a>0$, $\lambda >0$, $k>\kappa_0$, }
\end{align*}
where  $a\lambda>0$, $\rho_0=\left\{\frac{a(m+2)}{ \lambda}\right\}^{1/(m+2)}$ and $\kappa_0=\frac{\lambda}{m}\left\{\frac{a(m+2)}{ \lambda}\right\}^{2/(m+2)}$.

The above three conditions are treated in Lemmas below.
The proofs of Lemma \ref{5p6} and \ref{5p7} are similar to those of Lemma \ref{26} and \ref{comple22}.

\begin{lemma}  \label{5p6}
If $a >0, \lambda\leq 0$, i.e. if {\rm (IV.1)} is satisfied,
  then
  a positive solution $p$ of (\ref{mu2const1}) is defined on $\mathbb{R}$ and  $p^{'}(s)$ has a unique root. And  the space $(g,w)$ of  Lemma  \ref{LL61}   is defined on $\mathbb{R}^3$ and complete.
\end{lemma}
\begin{lemma} \label{5p7}
 Suppose $a <0, \lambda< 0, p(0)> \rho_0 $.

If $k =\kappa_0$,  then
  a positive solution $p$ of {\rm (\ref{mu2const1})} is defined on $\mathbb{R}$ and  $p^{'}(s)$ does not have a root. And  the space $(g,w)$ of  Lemma  {\rm \ref{LL61}}   is defined on $\mathbb{R}^3$ and complete.

If $k < \kappa_0$,    then
  a positive solution $p$ of {\rm (\ref{mu2const1})} is defined on $\mathbb{R}$ and  $p^{'}(s)$ has a unique root. And  the space $(g,w)$ of  Lemma  {\rm \ref{LL61} }  is defined on $\mathbb{R}^3$ and complete.
\end{lemma}

\begin{lemma} \label{5p8}
 If $a >0, \lambda> 0, k > \kappa_0$ , i.e. if {\rm (V)} is satisfied,
 then $(g,w)$ of  Lemma {\rm \ref{LL61}} is incomplete.
\end{lemma}

\begin{proof} By Proposition 2.6 in \cite{Ko}, a positive solution $p$ of {\rm (\ref{mu2const1})} is on $\mathbb{R}$,  non-constant and  periodic.
The function $G(p):=k- \frac{\lambda}{m+2}p^2-\frac{2a}{m}p^{-m}$ is decreasing on $p > \rho_0 $,  increasing on  $0<p <\rho_0 $ and  $G( \rho_0) = k- k_0 > 0$. For $p >0$,  $G(p)=0$ has exactly two solutions, say $p_1< p_2$.
  As $(p^{'}(s))^2 = G(p(s))$, at any point $s_0$ with $p^{'}(s_0)=0$,  we have $ G(p(s_0))=0$. Then $p(s_0)= p_1$  or
 $p(s_0)= p_2$. These $p_1, p_2$ are the minimum and maximum of $p(s)$.  There are $s_1, s_2 \in \mathbb{R}$ so that
  $p(s_1) = p_1$ and $p(s_2) = p_2$ and $(s_1, s_2)$ contains no root of $p^{'}$.
  We should consider $g= p^2(x_3)dx_1^2 +     (p^{'})^2 dx_2^2+ dx_3^2 $
when $   s_1 \leq  x_3 \leq s_2 $.
If $(g, w)$ is going to be complete, by section 1.4.1 of \cite{Pe}, we need;
$p^{''} (s_1) = -p^{''}(s_2) \neq 0 $. For this, (\ref{et07w})
gives
\begin{eqnarray} \label{a33}
\frac{\lambda}{m+2} (p_1+ p_2)=ap_1^{-m-1} + ap_2^{-m-1}
\end{eqnarray}
  And (\ref{mu2const1}) gives
$\frac{\lambda}{m+2}p_i^2+\frac{2a}{m}p_i^{-m}=k$ for $i=1,2$, from which we have

$\frac{\lambda}{m+2}( p_1+ p_2)  -\frac{2a}{m p_1 p_2} (p_1^{-m+1}+ p_1^{-m+2}p_2^{-1} + \cdots +p_2^{-m+1}) =0.$ Then,

$\frac{\lambda}{m+2} ( p_1+ p_2) p_1^{m+1}p_2^{m+1}  -p_1p_2\frac{2a}{m } (p_1^{m-1}+ p_1^{m-2}p_2 + \cdots +p_2^{m-1}) =0.$

\noindent From (\ref{a33}), $a(p_1^{m+1} + p_2^{m+1})-p_1p_2\frac{2a}{m } (p_1^{m-1}+ p_1^{m-2}p_2 + \cdots +p_2^{m-1}) =0$. So, $m(p_1^{m+1} + p_2^{m+1})-2(p_1^{m}p_2+ p_1^{m-1}p_2^2 + \cdots +p_1 p_2^{m}) =0$.

Factorizing this, we get $(p_1- p_2)^2 \{  \sum_{k=1}^m k(m-k+1) p_1^{k-1} p_2^{m-k}  \}=0$. This is impossible
as we require $p_1, p_2 >0$ and  $p_1 \neq p_2$.
So, $g$ cannot be complete near $x_3= s_1$ and $x_3= s_2$.

\end{proof}

\begin{prop} \label{pr56}
 Let $(M,g,w)$ be a complete $(\lambda,3+m)$-Einstein manifold such that there are exactly two distinct Ricci-eigen values on an open dense subset and that $\nabla w$ is not a Ricci-eigen vector field and  $E_1 \mu_2 = 0$ for
an adapted frame field $\{ E_i\}$.

Then in some coordinates $(x_1,x_2,x_3)$, $(g,w)$ can be as follows:
\begin{align*}
  g= (p(x_3))^2dx_1^2 +     (p^{'})^2 dx_2^2+ dx_3^2,
\end{align*}
where $p(x_3)$ is a non-constant positive solution on $\mathbb{R}$ of
\begin{equation} \label{mu2const1v}
(p')^2+\frac{\lambda}{m+2}p^2+\frac{2a}{m}p^{-m}=k.
\end{equation}
for constants $a \neq 0, k$ under one of the following conditions;

{\rm (ii-1)} $a>0$, $\lambda\leq 0$,

{\rm (ii-2)} $a<0$, $\lambda<0$, $p(0)>\rho_0$, $k\leq \kappa_0$.

\noindent  And  $w = {\tau}(x_1) p(x_3)$, where ${\tau}(x_1)$ is a non-constant positive function satisfying
\begin{equation} \label{mu2const1v5}
 ({\tau}')^2=-k \cdot \tau^2+\frac{\mu}{m-1}.
\end{equation}

Conversely, any space $(g, w)$ as in the above  satisfies $\lambda_1\neq\lambda_2=\lambda_3$,  $E_1\mu_2=0$ and {\rm (\ref{hpwdef})}.

\end{prop}

\section{When $\frac{\nabla w}{|\nabla w |} = E_1$}
 Now we shall study the case that $ \nabla w$ is a Ricci-eigen vector field.
In this section we assume that $\nabla w$ is parallel to $E_1$ for an adapted frame field $\{   E_i \}$. We may set
  $\frac{\nabla w}{|\nabla w|}=  E_1$. As $\frac{\nabla w}{|\nabla w|}=  E_1$, we have $\nabla_{E_2} w=  \nabla_{E_3} w= 0$.
In the next section we treat the other case that $\nabla w$ is orthogonal to $E_1$.

 This section includes 3-d locally conformally flat (non-Einstein) m-quasi Einstein manifolds, as they have the property \cite{CMMR} that $ \nabla w$ is Ricci-eigen, $\frac{\nabla w}{|\nabla w |} = E_1$ and  $ \lambda_1 \neq   \lambda_2 = \lambda_3   $.
The next lemma can be proved by standard argument;  see \cite{CMMR} or \cite[Lemma 2.3]{Ki}.

\begin{lemma} \label{threesolb1}
Let  $(M^3, g, w)$ be  a $(\lambda,3+m)$-Einstein manifold with  $ \lambda_1 \neq   \lambda_2 = \lambda_3   $ on an open subset of $ \{ \nabla w \neq 0  \}$. Let $c$ be a regular value of $w$ and $\Sigma_c= \{ x | \ w(x) =c  \}$. If $\frac{\nabla w}{|\nabla w|}=  E_1$ for an adapted frame field $\{   E_i \}$, then the followings hold.

{\rm (i)} $R$ and $ |\nabla w|^2$ are constant on a connected component of $\Sigma_c$.

{\rm (ii)} There is a function $s$ locally defined with   $s(x) = \int  \frac{   d w}{|\nabla w|} $, so that

$ \ \ \ \ ds =\frac{   d w}{|\nabla w|}$ and $E_1 = \nabla s$.

{\rm (iii)} $\nabla_{E_1} E_1=0$.

{\rm (iv)} $R({E_1, E_1})$ and  $R({E_2, E_2})$ are constant on a connected component of $\Sigma_c$, and so depend on the local variable  $s$ only.

{\rm (v)}  Near a point in $\Sigma_c$, the metric $g$ can be written as

$\ \ \ g= ds^2 +  \sum_{i,j >  1} g_{ij}(s, x_2, x_3) dx_i \otimes dx_j$, where
    $x_2, x_3$ is a local coordinates system on $\Sigma_c$.

{\rm (vi)} $\nabla_{E_i}E_1=\zeta(s) E_i, \textrm{$i=2,3$ with }\zeta(s)=\frac{w(\lambda_2-\lambda)}{m|\nabla w|}$.

\end{lemma}
\begin{proof}
By assumption, for $i=2,3$,  $R(\nabla w, E_i) =0$ and Lemma \ref{kklem} (ii) gives $E_i(R) =0$. Equation {\rm (\ref{hpwdef})} gives $E_i(|\nabla w|^2) =0$. We can see $d( \frac{dw}{|dw |} )=0$.
$\langle  \nabla_{E_1} E_1 , E_1 \rangle =0$ is trivial. We can get
$\langle  \nabla_{E_1} E_1 , E_i \rangle =\langle  \nabla_{E_1} (\frac{\nabla w}{|\nabla w |} ), E_i \rangle =0$ from  {\rm (\ref{hpwdef})}. (i), (ii) and (iii) are proved.
As $\nabla w$ and the level surfaces of $w$ are perpendicular, we get (v).

 Lemma \ref{kklem} (ii) gives
 $ \   \frac{w}{2(m-1)}E_1 R  = -R(E_1,E_1 )|\nabla w |+ \frac{(n - 1)\lambda - R }{ m-1}  |\nabla w | $.
From this,
$   -\frac{w}{2 |\nabla w |}E_1 R +2\lambda =  m\lambda_1  + 2\lambda_2$, the left hand side of which  is a function of $s$ only, as is $E_1(R) = \frac{dR}{ds}$. Since $R= \lambda_1  + 2\lambda_2$ is a function of $s$ only, we have that
$ \lambda_1$ and $\lambda_2$ both depend on $s$ only. We get (iv).
{\rm (\ref{hpwdef})} can be used to prove (vi).
\end{proof}

\begin{lemma} \label{claim112b3}
Let  $(M,g,w)$ be  a $(\lambda,3+m)$-Einstein manifold.
Suppose that $ \lambda_1  \neq \lambda_2=  \lambda_3$ and  $\frac{\nabla w}{|\nabla w|}=  E_1$ for an adapted frame field $\{ E_i\}$,
on an open subset $U$ of $\{ \nabla w \neq 0  \}$.

\smallskip
Then near any point of $U$, there exists local coordinates $(s, x_2, x_3)$  such that $\nabla s= \frac{\nabla w }{ |\nabla w |}$ and $g$ can be written as
\begin{equation} \label{mtr1a3}
g= ds^2 +      h(s)^2 \tilde{g},
\end{equation}
 where  $h:=h(s)$ is a smooth function and
 $\tilde{g}$ is (a pull-back of) a Riemannian metric of constant curvature on a $2$-dimensional domain with $x_2, x_3$ coordinates.
In particular,  $g$ is locally conformally flat.

 %We  get   $E_1 =\frac{\partial }{\partial s} $ and  $E_2 =\frac{1}{p} \frac{\partial }{\partial t} $.
 %, $E_3 =\frac{1}{h} e_3 $ and $E_4 =\frac{1}{h} e_4$, where $e_3$ and $e_4$ form an orthonormal frame field of $\tilde{g}$.

\end{lemma}

 \begin{proof}
For the metric $g$  of Lemma \ref{threesolb1} (v), one easily gets $E_1 =\frac{\partial }{\partial s} $.
We write $\partial_{1}:=\frac{\partial }{\partial s}$ and $\partial_{1}:=\frac{\partial }{\partial x_i}, i=2,3$ .

We consider the second fundamental form $\tilde{ h}$ of a leaf for $E_{23}$ with respect to $E_1$;
$\tilde{ h}  ( u , u ) =  -  \langle \nabla_{u} u ,  E_1\rangle  $. As the leaf is totally umbilic by Lemma  \ref{derdlem} {\rm (ii)}, $\tilde{ h} ( u , u ) =    \eta \cdot g( u , u) $ for some function  $\eta$ and any $u$ tangent to a leaf.
Then, $\tilde{ h} (E_2, E_2 ) = -  \langle \nabla_{E_2} E_2 ,  E_1\rangle = \eta=   \zeta   $,  which is a function of $s$ only by  Lemma \ref{threesolb1} (vi).

For $i, j \in \{ 2,3 \}$,
\begin{eqnarray*}
\zeta  g_{ij}& =\tilde{ h} ( \partial_i , \partial_{j} )  =   -  \langle \nabla_{\partial_i}  \partial_{j} ,   \frac{\partial }{\partial s}\rangle =  -  \langle\sum_k \Gamma^{k}_{i{j}}  \partial_k ,   \frac{\partial }{\partial s}  \rangle \ \ \ \ \ \ \  \\
& =  - \sum_k \langle  \frac{1}{2} g^{kl}( \partial_i g_{lj} +\partial_{j} g_{li} - \partial_l g_{ij} )\partial_k ,  \frac{\partial }{\partial s} \rangle       = \frac{1}{2} \frac{\partial }{\partial s} g_{i{j}}.
\end{eqnarray*}
 So,  $\frac{1}{2} \frac{\partial }{\partial s} g_{i{j}} =   \zeta g_{ij}$. Integrating it, for $i, j \in \{ 2,3 \}$, we get $ g_{ij} = e^{C_{ij}} h(s)^2$. Here the function $h(s)>0$ is independent of $i,j$ and each function $C_{ij}$ depends only on $x_2, x_3$.

 Now  $g$ can be  written as $g= ds^2 +     h(s)^2 \tilde{g} $, where $\tilde{g}$ can be viewed as a Rimannian metric in a domain of $(x_2, x_3)$-plane.

From Gauss-Codazzi equation,
$R^{g} = R^{\tilde{g}} + 2 Ric^{g}(E_1,E_1)+ \|h\|^2 - H^2$.  As all others are constant on a hypersurface of $w$, so is $R^{\tilde{g}}$. Therefore each hypersurface has constant curvature. Thus $\tilde{g}$ has constant curvature and
 $g$ is locally conformally flat.
 \end{proof}
For the metric of (\ref{mtr1a3}), the equation
 (\ref{hpwdef}) is equivalent to;
\begin{align*}
 m \frac{w^{''}}{w}&=-2 \frac{h^{''}}{h}-\lambda ,\\
  m \frac{w^{'}}{w}\frac{h^{'}}{h}&=  -\frac{h^{''}}{h} - (\frac{h^{'}}{h})^2 + \frac{k}{h^2}  -\lambda.
\end{align*}
 where $k$ is the sectional curvature of $\tilde{g}$.

\section{ When $\frac{\nabla w}{|\nabla w |} = E_3$ }

We still assume  $ \nabla w$ is a Ricci-eigen vector field on a domain and  $ \lambda_1=R(E_1, E_1) \neq  \lambda_2  =  \lambda_3$.
We assume that $\nabla w$ is orthogonal to $E_1$. We may choose $E_2, E_3$ so that $ E_3=\frac{\nabla w}{|\nabla w |}$. Then $E_1w=E_2w=0$.
 The next lemma can be proved similarly as Lemma  \ref{threesolb1}.

\begin{lemma} \label{threesolb12}
Let  $(M^3, g, w)$ be  a $(\lambda,3+m)$-Einstein manifold.
 Assume $ \lambda_1=R(E_1, E_1) \neq  \lambda_2  =  \lambda_3$ and $ E_3=\frac{\nabla w}{|\nabla w |}$ on an open subset of $ \{ \nabla w \neq 0  \}$.
 Let $c$ be a regular value of $w$ and $\Sigma_c= \{ x | \ w(x) =c  \}$. Then the followings hold.

{\rm (i)} $R$ and $ |\nabla w|^2$ are constant on a connected component of $\Sigma_c$.

{\rm (ii)} $\nabla_{E_3} E_3=0$.

{\rm (iii)} $\lambda_1$ and  $\lambda_2$ are constant on a connected component of $\Sigma_c$.

\end{lemma}

\begin{lemma} \label{claim112na2}
Under the same hypothesis as Lemma {\rm \ref{threesolb12}},  the followings hold;

 \medskip
 $\nabla_{E_i}  E_3= \zeta_i E_i$ for $i=1,2$, with $ \zeta_i= \frac{ w}{m} \frac{ R(E_i, E_i) -  \lambda }{  | \nabla w |}  $.

$  \nabla_{E_1}  E_1 = -\zeta_1  E_3, \ $ $  \ \ \ \ \ \ \ \nabla_{E_1}  E_2=0$,  $ \ \ \ \  \nabla_{E_2}  E_1=  0 $.

$ \nabla_{E_2}  E_2 =  -\zeta_2 E_3$,  $ \ \ \ \ \nabla_{E_3}  E_1= 0, \ $ $\ \ \  \nabla_{E_3}  E_2= 0, \ \ \ \  $ $ \nabla_{E_3}  E_3=0$.

$[E_3, E_2]= -\zeta_2 E_2, \ \ \ \ \ \ $  $ \ [E_3, E_1]= -\zeta_1 E_1, \   $  $ \  \ \ \ \ \ [E_2, E_1]=0.$

\medskip
$\zeta_1$ and $ \zeta_2 $ are constant on a connected component of $\Sigma_c$.

\medskip
 The three {\rm 2}-d distributions $E_{ij}$ spanned by $E_i$ and $E_j$, $ i\neq j $ are integrable.
\end{lemma}

\begin{proof}
Recall that for $j=1,2$,  the eigenvalues of our Codazzi tensor ${\mathcal C}$ are $ \mu_j=w^{m+1}   \lambda_j  -\frac{m}{2(m-1)} R w^{m+1} +\frac{1}{ m-1} \lambda  w^{m+1} $. By Lemma \ref{threesolb12}, $E_i (\lambda_j)=0$ and  $E_i (\mu_j)=0$  for $i, j \leq 2$.
By Lemma \ref{derdlem} (i) and $\mu_2=\mu_3$, we have $\langle\nabla_{E_1}  E_1, E_2  \rangle=0$, $\langle\nabla_{E_2}  E_2, E_1  \rangle=\langle\nabla_{E_3}  E_3, E_1  \rangle=0$ and $ \langle\nabla_{E_3}  E_2, E_1\rangle=  \langle\nabla_{E_2}  E_3, E_1\rangle= 0$.

Other formulas can be easily obtained from {\rm (\ref{hpwdef})}.  The constancy of  $\zeta_i $  on a connected component of $\Sigma_c$ is from Lemma \ref{threesolb12}.
Integrability of $E_{ij}$ is from bracket formulas.
\end{proof}

By Lemma \ref{claim112b8},  the metric $g$ can be written as
$g=     g_{11}dx_1^2 + g_{22}dx_2^2   + g_{33}dx_3^2,$
where $g_{ij}$ are functions  of $(x_1, x_2, x_3)$ and  $E_i =  \frac{1}{\sqrt{g_{ii}}}\frac{\partial }{\partial x_i} $ for $i=1,2,3$.
By (\ref{e1f545}), $\nabla_{E_3}E_3=0$ tells us that $\partial_1g_{33}=\partial_2g_{33}=0$. Thus we can
 replace  $g_{33}dx_3^2$  by $dx_3^2$. In this new coordinates  $(x_1, \ x_2, \ x_3)$ we write
 $\partial_i := \frac{\partial}{\partial x_i}$.
From $\langle\nabla_{E_2}  E_2,  E_1\rangle = 0$, we get  $ \partial_1 g_{22}=0$. Similarly,  $ \partial_2 g_{11}=0$. We may write for some functions $p$ and $h$,
\begin{equation} \label{ggg21}
g=   ( p(x_1, x_3))^2 dx_1^2 +(h(x_2, x_3))^2 dx_2^2  + dx_3^2.
\end{equation}
With $E_2 =\frac{1}{h(x_2, x_3)}\partial_2,$ $E_3 =\partial_3$ and $E_1 =\frac{1}{p(x_1, x_3)}\partial_1$, we get
 $ \zeta_1(x_3)=  \langle\nabla_{E_1}E_3, E_1\rangle=\frac{\partial_3g_{11}}{2g_{11}}= \frac{\partial_3 p}{p} $ and $ \zeta_2(x_3)= \langle\nabla_{E_2}E_3, E_2\rangle= \frac{\partial_3h}{h} $. Here $\zeta_i$ depend only on $x_3$, because they  are constant on a connected component of $\Sigma_c$ which is tangent to the span of $\partial_1$ and  $\partial_2$.

Integrating them, we have $h(x_2, x_3) = h_1(x_2) \cdot h_2(x_3) $ and $p(x_1, x_3) = p_1(x_1) \cdot p_2(x_3) $  with positive functions $h_1, h_2, p_1, p_2$.

Replacing $h_1^2 dx_2^2$ by $dx_2^2$  and  $p_1^2 dx_1^2$  by $dx_1^2$,
we may write
 \begin{equation} \label{ggg22}
g= p^2(x_3) dx_1^2    + h^2(x_3)dx_2^2 +d x_3^2,
\end{equation}
with
\begin{equation} \label{ggg216}
\zeta_1(x_3)= \frac{d}{dx_3}( \ln p) \  \  \ \ \ \ \  \   \ \zeta_2(x_3)= \frac{d}{dx_3} ( \ln h).
\end{equation}
The condition $\lambda_2= \lambda_3$ is equivalent to $R_{2112} =R_{3113}$ where
$R_{3113} =- \frac{p^{''}}{p}$ and $R_{2112}= -\frac{p^{'} h^{'}}{ph}$.
We get
\begin{align}
  \frac{p^{''}}{p}= \frac{p^{'} h^{'}}{ph}. \label{relph}
\end{align}
\begin{lemma} \label{cl02}
For the metric in {\rm (\ref{ggg22})}, if $p$ is not a constant, then the metric $g$ becomes
\begin{align*}
  g=(p(x_3))^2dx_1^2+(p'(x_3))^2dx_2^2+dx_3^2,
\end{align*}
where $p$ is a solution of $(p')^2+\frac{\lambda}{m+2}p^2+\frac{2a}{m}p^{-m}=0$ for a constant $a\neq 0$. And the potential function $w=c_1p$ for a constant $c_1>0$. The constant $\mu$ of {\rm (\ref{consteq})} is zero.
Conversely, any $(g, w)$ as above  satisfies $\lambda_1\neq\lambda_2=\lambda_3$,  $\frac{\nabla w}{|\nabla w |} = E_3$, $\mu=0$
 and {\rm (\ref{hpwdef})}.

\smallskip
The metric $g$ is complete precisely when $a>0$, $\lambda<0$.
\end{lemma}
\begin{proof}
 If $p$ is not a constant, then we get $p'=Ch$ from (\ref{relph}).
  As $p^{'} = C h$, after replacing $\frac{
x_2}{C}$ by new $x_2$,  the metric of {\rm (\ref{ggg22})} can be written as
\begin{equation} \label{ggg24}
 g= (p(x_3))^2 dx_1^2    + (p^{'}(x_3))^2  dx_2^2 +d x_3^2.
\end{equation}
From now the proof will be similar to that of Lemma \ref{LL61}.
 Ricci curvature components are $R_{11} :=R(E_1, E_1) = -2  \frac{p^{''}}{p}$,
$R_{22}=R_{33}  =- \frac{p^{''}}{p} - \frac{p^{'''}}{p^{'}}  $.  Use Lemma \ref{claim112na2} and (\ref{ggg216}) to get the nontrivial components of  the equation  $\nabla dw=\frac{w}{m}(Rc-\lambda g)$ as below.
\begin{align}
  \frac{p'}{p}w'=\frac{w}{m}(-2\frac{p''}{p}-\lambda)   \ \ \ \ \ \ \ \ \label{wfirst}\\
 \frac{p''}{p'}w'=\frac{w}{m}(-\frac{p''}{p}-\frac{p'''}{p'}-\lambda) \  \ \label{wsecond}\\
 w''=\frac{w}{m}(-\frac{p''}{p}-\frac{p'''}{p'}-\lambda)   \ \ \ \label{wthird}
\end{align}
By (\ref{wsecond}) and (\ref{wthird}), $w''=  \frac{p''}{p'}w'$ and we get $w=c_1 p+b_1$ for some constants $c_1\neq0$ and $b_1$. Then by (\ref{wfirst}),
\begin{align}
 mc_1(p')^2+(c_1p+b_1)(2p''+\lambda p)=0. \label{pdouble}
\end{align}
And (\ref{wthird}) gives
\begin{align}
 c_1 m p^{''}+( c_1 p+b_1 )(\frac{p''}{p}+\frac{p'''}{p'}+\lambda)=0. \label{pdouble7}
\end{align}

Differentiating (\ref{pdouble}) and comparing with (\ref{pdouble7}), we get $b_1p'(2p''+\lambda p)=0$. But if $2p''+\lambda p=0$, then $p'=0$ by (\ref{pdouble}). Thus $b_1=0$.

Multiplying $p^{m-1}p'$ to (\ref{pdouble}) and integrating, $p^m(p')^2+\frac{\lambda}{m+2}p^{m+2}+\frac{2a}{m}=0$ for a constant $a$. And
\begin{align}
  (p')^2+\frac{\lambda}{m+2}p^{2}+\frac{2a}{m}p^{-m}=0. \label{pm}
\end{align}
  Using (\ref{pm}) the Ricci eigenvalues can be computed as $ \lambda_1=m(\frac{p'}{p})^2+\lambda$, $\lambda_2=-\frac{m^2}{2}(\frac{p'}{p})^2-\frac{\lambda}{2}(m-2)$.
Assigning these values to (\ref{consteq}), we get $\mu=0$.

If $a=0$ then one can get $R_{11} =R_{22}$, a contradiction. So, $a \neq 0$.  Note that the above ODE is also identical to the Kobayahis's equation (2.2) in \cite{Ko} if the correspondence is $k \leftrightarrow 0, \ n-2 \leftrightarrow m, \ \frac{R}{n-1} \leftrightarrow \lambda $. As $a \neq 0$ and $k=0$, among Kobayashi's list we get a shorter list of
I.1, III.1, III.3 and IV.1. By a similar argument as in Sections 4,  the conditions  I.1, III.1 and III.3 do not hold a complete space.
As in the proof of Lemma \ref{26} we can show that the IV.1 case  of $a>0$, $\lambda\leq 0$ yields a complete space $(g, w)$. Actually $\lambda<0$ from (\ref{pm}).

\end{proof}

\begin{lemma} \label{L74}
  For the metric in {\rm (\ref{ggg22})}, if $p$ is a constant, then $\lambda=0$ and the metric $g$ becomes
   %there exist coordinates $(x_1,x_2,x_3)$ such that the metric $g$ is
\begin{align*}
  g=dx_1^2+(w'(x_3))^2dx_2^2+dx_3^2,
\end{align*}
and the  potential function $w$ satisfies $w=w(x_3)$ and
\begin{align}
  (w')^2+\frac{2a}{m-1}w^{1-m}=\frac{\mu}{m-1}, \ \ {\rm  for} \ {\rm  a} \ {\rm  constant} \  a\neq0.
\end{align}
Conversely, any $(g, w)$ as above  satisfies $\lambda_1\neq\lambda_2=\lambda_3$,  $\frac{\nabla w}{|\nabla w |} = E_3$ and {\rm (\ref{hpwdef})} with $\lambda=0$.

\medskip
The space $(g, w)$ is complete precisely when $a>0$,  in which case $\mu >0$.
\end{lemma}

\begin{proof}
 If $p$ is a constant, by (\ref{ggg216}) $\zeta_1=0$. And $R_{11}=0$. By $\zeta_1=\frac{w}{m|\nabla w|}(\lambda_1-\lambda)$, we get $\lambda=0$. From (\ref{ggg22}), $g$ can be regarded as a product metric $g=dx_1^2+\tilde{g}$ where $\tilde{g}=h(x_3)^2dx_2^2 +dx_3^2$.

From Lemma \ref{claim112na2} and (\ref{ggg216}),  the nontrivial components of  the equation  $\nabla dw=\frac{w}{m}(Rc-\lambda g)$ are as below.
\begin{align}
  w''=-\frac{w}{m}\frac{h^{''}}{h} \  \ \label{wsecond4}\\
  \frac{h'}{h}w'=-\frac{w}{m}\frac{h^{''}}{h} \label{wthird4}
\end{align}
 From the above, $w''=
  \frac{h'}{h}w'$ and we get $h= c w^{'}$ for a nonzero constant $c$. From (\ref{wsecond4}) $mw^{'} w^{''} + w w^{'''}=0$, which can be integrated to $w^m w^{''} = a$  for a constant $a.$  Integrating this, for a constant $\mu_0$ we get $(w')^2+\frac{2a}{m-1}w^{1-m}=\frac{\mu_0}{m-1}.$ From this and $R_{22}=R_{33} = -\frac{w^{'''}}{w^{'}}$, one can see that $\mu_0$ is the constant $\mu$ of (\ref{consteq}). So we write;
 \begin{align} \label{ww2}
  (w')^2+\frac{2a}{m-1}w^{1-m}=\frac{\mu}{m-1}.
\end{align}
And if $a=0$, then $R_{11}=R_{22}$, a contraduction. So, $a \neq 0$.

After replacing $c x_2$ by new $x_2$,
the metric $g$ can written as  $g=dx_1^2+w'(x_3)^2dx_2^2+dx_3^2$.

(\ref{ww2}) is  identical to the Kobayashi's equation (2.2) in \cite{Ko} if the correspondence is $k \leftrightarrow \frac{\mu}{m-1}, n-2 \leftrightarrow m-1, R \leftrightarrow 0 $. As $a \neq 0$ and $R=0$, among Kobayashi's list we get a shorter list of
III.1 and  IV.1. As in Lemma \ref{LL3k}, the condition III.1 cannot hold a complete space.

By a similar argument as in Lemma \ref{26} we can show that the IV.1 case  of $a>0$ yields a complete space $(g, w)$.
In this case, $\mu>0$ from (\ref{ww2}).
\end{proof}

We combine Lemma \ref{cl02} and \ref{L74} and state
\begin{prop} \label{cl025}
Let  $(M,g,w)$ be  a $(\lambda,3+m)$-Einstein manifold.
 Assume $ E_3=\frac{\nabla w}{|\nabla w |}$ for an adapted frame field $\{E_i\}$ on an open subset of $ \{ \nabla w \neq 0  \}$.

Then there exist coordinates $(x_1, x_2, x_3)$ locally in which the space $(g, w)$ takes one of the following forms;

\medskip
{\rm (i-3)} $g=dx_1^2+(w'(x_3))^2dx_2^2+dx_3^2$ and the potential function $w=w(x_3)$ is a non-constant positive solution of  $(w')^2+\frac{2a}{m-1}w^{1-m}=\frac{\mu}{m-1}$ for a constant $a\neq0$. The constant $\lambda=0$.
 The space $(g,w)$ is complete precisely when $a>0$. In this complete case $\mu >0$.

\medskip
{\rm (ii-3)} $g=(p(x_3))^2dx_1^2+(p'(x_3))^2dx_2^2+dx_3^2$
where $p:=p(x_3)$ is a non-constant positive solution of $(p')^2+\frac{\lambda}{m+2}p^2+\frac{2a}{m}p^{-m}=0$ for a constant $a\neq0$. And $w=c_1p$ for a constant $c_1>0$. And $\mu=0$.  The space $(g,w)$ is complete precisely when $a>0$, $\lambda < 0$.

 Conversely, any space $(g, w)$ in {\rm (i-3)} or {\rm (ii-3)}   satisfies $\lambda_1=R(E_1, E_1) \neq\lambda_2=\lambda_3$,  $\frac{\nabla w}{|\nabla w |} = E_3$ and {\rm (\ref{hpwdef})}.

\end{prop}

\section{Warped product Einstein manifolds and proofs of theorems}

In this section we explain the warped product Einstein manifolds associated to the complete $(\lambda,3+m)$-Einstein manifolds  with $\lambda_1\neq \lambda_2 =\lambda_3$.
Recall that two dimensional complete  $(\lambda,2+m)$-Einstein manifolds and their  warped product Einstein manifolds are descibed in 9.118 of Besse \cite{Be}. Some four dimensional work can be found in \cite{SJ}.
By comparing these, we find that all these works are closely related to  Kobayashi's static spaces.

We discuss the complete warped product Einstein manifolds out of Proposition \ref{pr47}, \ref{pr56} and \ref{cl025}.

\subsection{Complete spaces  out of  Proposition \ref{pr47} } \label{se1}
%A complete m-QE space $(g,w)$ in Proposition \ref{pr47} is of the form $g = dx_1^2+\eta^2(x_1)p'(x_3)^2dx_2^2+\eta^2(x_1)dx_3^2$ with $(\partial_1\eta)^2+\frac{\lambda}{m+2}\eta^2=\rho  \leq 0$, $(p')^2+\rho p^2+\frac{2a}{m-1}p^{1-m}=\frac{\mu}{m-1}$ and $w= \eta p$.The associated warped product Einstein metric can be written as$g_E=dx_1^2+\eta^2(p'(x_3)^2dx_2^2+dx_3^2 + p^2 g_F)$.

A complete  $(\lambda,3+m)$-Einstein manifold $(g,w)$ in Proposition \ref{pr47} is of the form

  $g=dx_1^2+(\eta(x_1))^2 (p'(x_3))^2dx_2^2+(\eta(x_1))^2dx_3^2$ and $w= \eta \cdot p$ satisfying  (\ref{koeq2}) and (\ref{koeq2v}). We have the associated warped product Einstein metric as follows; for a complete Einstein manifold $(F, g_F)$ with $Rc_{g_F} =  \mu g_F$,
\begin{equation} \label{440}
g_E  = dx_1^2+(\eta(x_1))^2 (p'(x_3)^2dx_2^2+dx_3^2+ p^2 g_F).
\end{equation}
We observe that  $g_0:=(p'(x_3))^2dx_2^2+dx_3^2+ p^2 g_F$ is an Einstein metric with $Rc_{g_0} =  \rho  (m+1) g_0$.
In fact, the metric $dx_3^2+ p^2 g_F$ satisfying (\ref{koeq2}) is a static space with static potential function $p^{'}$ by Kobayashi's study \cite{Ko}, so $g_0$ is Einstein.
 Now $\eta$ satisfy (\ref{koeq2v}), so one can reassure that
 $g_E  = dx_1^2+\eta^2g_0$ is  an Einstein metric with $Rc_{g_E} =  \lambda  g_E$.

\medskip

To have a complete space, we are under one of the following cases;

{\rm (i-1)}  $\lambda<0$, $a>0$, $\rho  \leq 0$

{\rm (i-2)} $\lambda<0$, $a<0$, $\rho <0$, $\mu\leq (m+1)\rho (\frac{a}{ \rho })^{\frac{2}{m+1}}$, $p(0)>(\frac{a}{ \rho })^{\frac{1}{m+1}}$.

\medskip

By the proof of Lemma \ref{26}, under {\rm (i-1)} $g_E$ can be smoothly defined on $\mathbb{R}^3\times F=\{ (x_1, x_2, x_3, z) \ |  \  x_1 \in \mathbb{R}, x_2 \in [0, \frac{2 \pi}{p{''}(0)}], x_3 \geq 0, z \in F   \}$ where $(x_3, x_2)$ is regarded as the polar coordinates of a plane.   The condition {\rm (i-1)} does not contain a restriction on the sign of $\mu$.

By the proof of Lemma \ref{comple22}, under {\rm (i-2)} with $\mu< \kappa_0=(m+1)\rho (\frac{a}{ \rho })^{2/(m+1)}$,  $g_E$ is on $\mathbb{R}^3\times F$ where $(x_3, x_2)$, $ x_3 \geq 0$, is the plane-polar coordinates. But under {\rm (i-2)} with $\mu=\kappa_0$, $g_E$ is on $\mathbb{R}^3\times F$ where $(x_1, x_2, x_3)$ is the standard coordinates of $\mathbb{R}^3$.

\subsection{Complete spaces  out of  Proposition \ref{pr56} } \label{se2}

A complete  $(\lambda,3+m)$-Einstein manifold $(g,w)$ in Proposition \ref{pr56} is of the form
  $g =  p^2(x_3)dx_1^2 +     (p^{'})^2 dx_2^2+ dx_3^2$ and  $w = {\tau}(x_1) p(x_3)$,
 satisfying (\ref{mu2const1v}) and
(\ref{mu2const1v5}).
We have the associated warped product Einstein metric as follows; for a complete Einstein manifold $(F, g_F)$ with $Rc_{g_F} =  \mu g_F$,
\begin{equation} \label{440b}
g_E  = (p(x_3))^2(dx_1^2 + {\tau(x_1)}^2 g_F)+     (p^{'})^2 dx_2^2+ dx_3^2.
\end{equation}
The metric $\tilde{g}:=dx_1^2 + {\tau}^2 g_F$ is Einstein with $Rc_{\tilde{g}}= km\tilde{g}$ because of (\ref{mu2const1v5}).
 Then by (\ref{mu2const1v}),   $g_0:= p^2(x_3)(dx_1^2 + {\tau}^2 g_F)+ dx_3^2$ is a static space with  potential function $p^{'}$; see  \cite{Ko}.
 So one can reassure that
 $g_E  = g_0 +     (p^{'})^2 dx_2^2$ is  an Einstein metric with $Rc_{g_E} =  \lambda  g_E$.

 \medskip
To have a complete space, we are under one of the following cases;

(ii-1) $a>0$, $\lambda\leq 0$

(ii-2) $a<0$, $\lambda<0$, $p(0)>\left\{\frac{a(m+2)}{ \lambda}\right\}^{\frac{1}{m+1}}$, $k\leq \kappa_0= \frac{\lambda}{m}\left\{\frac{a(m+2)}{ \lambda}\right\}^{\frac{2}{m+1}}.$

\medskip
We shall see that depending on conditions, the underlying manifold of $g_E$ can vary:  $\mathbb{R}^{2} \times  \mathbb{S}^{m+1}$, $\mathbb{R}^{m+3}$ or $\mathbb{R}^3 \times F $.

Under {\rm (ii-1)}, as the above {\rm (i-1)} case, $(x_3, x_2), x_3 \geq 0,$ is regarded as the plane-polar coordinates;

\medskip
$\bullet \ $   If {\rm (ii-1)} holds and $k>0$, then $\mu >0$ and  we can write

 $g_E=p^2(dx_1^2 +\frac{\mu \sin^2( \sqrt{k} x_1 )}{(m-1)k}g_F)+     (p^{'})^2 dx_2^2+ dx_3^2$
with $Ric_{g_E} \leq 0$.  $g_F$ should be a round spherical metric and then $dx_1^2 +\frac{\mu \sin^2( \sqrt{k} x_1 )}{(m-1)k}g_F$ should have  positive constant curvature. So, $g_E$ is defined on $\mathbb{S}^{m+1} \times  \mathbb{R}^{2}$.
In particular, we have the Riemannian Schwarzschild Ricci flat metric when $Ric_{g_E}=0$.

\medskip

$\bullet \ $    If {\rm (ii-1)} holds and $k=0$, then $\mu >0$ and

 $g_E=p^2(dx_1^2 + \frac{\mu}{m-1}x_1^2g_F)  +  (p^{'})^2 dx_2^2+ dx_3^2$
with $Ric_{g_E} < 0$.  $g_F$ should be a round spherical metric and then   $dx_1^2 + \frac{\mu}{m-1}x_1^2g_F$ should have  zero constant curvature. So, $g_E$ is defined on $\mathbb{R}^{m+3}$.

\medskip

$\bullet \ $    If {\rm (ii-1)} holds and  $k<0$, then $\lambda<0$.
Depending on the sign of  $\mu$, we may write as below.
\begin{equation} \label{651}
g_E = p^2(dx_1^2+
\frac{\mu \sinh^2(\sqrt{-k}x_1)}{(m-1)(-k)}g_F) +     (p^{'})^2 dx_2^2+ dx_3^2  \ {\rm on} \ \mathbb{R}^{m+3}   \ {\rm if} \ \mu> 0.
\end{equation}
where $g_F$ should be a round spherical metric with $Rc_{g_F}  = -(m -1) k g_F$.
 \begin{equation} \label{652}
g_E = p^2(dx_1^2+e^{\pm 2\sqrt{-k} x_1 }g_F) +     (p^{'})^2 dx_2^2+ dx_3^2  \ {\rm on} \ \mathbb{R}^3\times F, \ {\rm if} \ \mu= 0.
\end{equation}
  \begin{equation} \label{653}
g_E = p^2(dx_1^2+ \frac{\mu \cosh^2(\sqrt{-k}x_1)}{(m-1)k}g_F) +     (p^{'})^2 dx_2^2+ dx_3^2 \ {\rm on} \ \mathbb{R}^3\times F, \  {\rm if} \ \mu< 0.
\end{equation}

$\bullet \ $   If {\rm (ii-2)} holds, then as $k<0$, ${\tau}$ can be defined on $\mathbb{R}$ with any sign of $\mu$. $Ric_{g_E} < 0$. The description of $g_E$ would be similar to (\ref{651})$\sim$(\ref{653}) although the behaviour of $p$ and roles of $x_2, x_3$ can be different.

\subsection{Complete spaces  out of  Proposition \ref{cl025} } \label{se3}
$   $

\medskip
$\bullet$  A complete  $(\lambda,3+m)$-Einstein manifold $(g,w)$ in Proposition \ref{cl025} {\rm (i-3)}  is of the form
 $g=dx_1^2+(w'(x_3))^2dx_2^2+dx_3^2$ with $(w')^2+\frac{2a}{m-1}w^{1-m}=\frac{\mu}{m-1}$ with $a>0$.
Here $\lambda=0$ and $\mu >0$.
Then $g_E=dx_1^2+w'(x_3)^2dx_2^2+dx_3^2 + w(x_3)^2 g_F$ on $\mathbb{R}^3\times F$ with $Ric_F>0$ and $Ric_{g_E}= 0$.
This class contains $\mathbb{R} \times Sc$, where $Sc$ denotes the Riemannian Schwarzschild metric.
The case 9.118 (a) of \cite{Be} corresponds to this class.

\medskip
$\bullet$ A complete  $(\lambda,3+m)$-Einstein manifold $(g,w)$ in Proposition \ref{cl025} {\rm (ii-3)}  is of the form  $g=(p(x_3))^2dx_1^2+(p'(x_3))^2dx_2^2+dx_3^2$
with $(p')^2+\frac{\lambda}{m+2}p^2+\frac{2a}{m}p^{-m}=0$ when $a>0$, $\lambda < 0$. And $w=c_1p$ for a constant $c_1>0$. And $\mu=0$.
Then $g_E=(p(x_3))^2 (dx_1^2 + g_F )+(p'(x_3))^2dx_2^2 +dx_3^2$ on $\mathbb{R}^3\times F$  with $Ric_F=0$.  $Ric_{g_E} < 0$. This can be viewed as the special case $k=\mu=0$ of (\ref{440b}).

\subsection{Proofs of Theorems}

Here we finish the proofs of  Theorem \ref{maint1} and \ref{maint2}   by summarizing most of the above discussions.

\smallskip
{\bf Proof of Theorem \ref{maint1}:}
 Lemma \ref{e1mu2not0} and Proposition \ref{cl025} (i-3) can be put together to yield Theorem \ref{maint1} (1), while
Lemma \ref{LL61} and Proposition \ref{cl025} (ii-3)
yield (2). Lemma \ref{claim112b3} gives (3).  We returned to m-QE manifolds via the transformation $w= e^{-\frac{f}{m}}$.

\medskip

{\bf Proof of Theorem \ref{maint2}:}
Now we put together the discussions in  subsections \ref{se1}$\sim$\ref{se3}.
Subsection \ref{se1} yields  (i-1) and (i-2) in Theorem \ref{maint2} .
The first case in  \ref{se3} yields (i-3).  And  \ref{se2} gives (ii-1) and (ii-2).
The second case in \ref{se3} is just absorbed into (ii-1).

\medskip
 Under some conditions on constants $\lambda, \rho, a,\mu,k$,
 one may get compact singular spaces of the form (1) or (2) in Theorem \ref{maint1}, but not smooth ones.
Meanwhile, concerning (3) of  Theorem \ref{maint1}, we recall
 B\"{o}hm's remarkable collection of compact warped product Einstein manifolds \cite{Bo} of the form {\rm (iii)} in  Theorem \ref{maint2}.

\end{document}